\documentclass[11pt]{amsart}
\usepackage{amsmath,amssymb,amscd}
\begin{document}
\theoremstyle{plain}
\newtheorem{thm}{Theorem}[section]
\newtheorem*{thm*}{Theorem}
\newtheorem{prop}[thm]{Proposition}
\newtheorem{prop-def}[thm]{Proposition/Definition}
\newtheorem*{prop*}{Proposition}
\newtheorem{lemma}[thm]{Lemma}
\newtheorem{cor}[thm]{Corollary}
\newtheorem*{conj*}{Conjecture}
\newtheorem*{cor*}{Corollary}
\newtheorem{defn}[thm]{Definition}
\newtheorem{cond}{Condition}
\theoremstyle{definition}
\newtheorem*{defn*}{Definition}
\newtheorem{rems}[thm]{Remarks}
\newtheorem*{rems*}{Remarks}
\newtheorem*{proof*}{Proof}
\newtheorem{rem}[thm]{Remark}
\newtheorem*{rem*}{Remark}
\newtheorem*{not*}{Notation}
\newcommand{\npartial}{\slash\!\!\!\partial}
\newcommand{\Heis}{\operatorname{Heis}}
\newcommand{\Solv}{\operatorname{Solv}}
\newcommand{\Spin}{\operatorname{Spin}}
\newcommand{\SO}{\operatorname{SO}}
\newcommand{\ind}{\operatorname{ind}}
\newcommand{\Index}{\operatorname{index}}
\newcommand{\ch}{\operatorname{ch}}
\newcommand{\rank}{\operatorname{rank}}
\newcommand{\abs}[1]{\lvert#1\rvert}
 \newcommand{\A}{{\mathcal A}}
        \newcommand{\D}{{\mathcal D}}\newcommand{\HH}{{\mathcal H}}
        \newcommand{\LL}{{\mathcal L}}
        \newcommand{\B}{{\mathcal B}}
        \newcommand{\K}{{\mathcal K}}
 \newcommand{\E}{{\mathcal E}}
\newcommand{\oo}{{\mathcal O}}
         \newcommand{\PP}{{\mathcal P}}
        \newcommand{\s}{\sigma}
\newcommand{\al}{\alpha}
        \newcommand{\coker}{{\mbox coker}}
        \newcommand{\p}{\partial}
        \newcommand{\dd}{|\D|}
        \newcommand{\n}{\parallel}
\newcommand{\bma}{\left(\begin{array}{cc}}
\newcommand{\ema}{\end{array}\right)}
\newcommand{\bca}{\left(\begin{array}{c}}
\newcommand{\eca}{\end{array}\right)}
\def\clsp{\overline{\operatorname{span}}}
\def\T{\mathbb T}
\def\Aut{\operatorname{Aut}}
\newcommand{\sr}{\stackrel}
\newcommand{\da}{\downarrow}
\newcommand{\tD}{\tilde{\D}}
        \newcommand{\R}{\mathbf R}
        \newcommand{\C}{\mathbf C}
        \newcommand{\h}{\mathbf H}
\newcommand{\Z}{\mathbf Z}
\newcommand{\N}{\mathbf N}
\newcommand{\tto}{\longrightarrow}
\newcommand{\ben}{\begin{displaymath}}
        \newcommand{\een}{\end{displaymath}}
\newcommand{\be}{\begin{equation}}
\newcommand{\ee}{\end{equation}}
        \newcommand{\bean}{\begin{eqnarray*}}
        \newcommand{\eean}{\end{eqnarray*}}
\newcommand{\nno}{\nonumber\\}
\newcommand{\bea}{\begin{eqnarray}}
        \newcommand{\eea}{\end{eqnarray}}
\def\cross#1{\rlap{\hskip#1pt\hbox{$-$}}}
        \def\intcross{\cross{0.3}\int}
        \def\bigintcross{\cross{2.3}\int}
\newcommand{\supp}[1]{\operatorname{#1}}
\newcommand{\norm}[1]{\parallel\, #1\, \parallel}
\newcommand{\ip}[2]{\langle #1,#2\rangle}
\setlength{\parskip}{.3cm}
\newcommand{\nc}{\newcommand}
\nc{\nt}{\newtheorem} \nc{\gf}[2]{\genfrac{}{}{0pt}{}{#1}{#2}}
\nc{\mb}[1]{{\mbox{$ #1 $}}} \nc{\real}{{\mathbb R}}
\nc{\comp}{{\mathbb C}} \nc{\ints}{{\mathbb Z}}
\nc{\Ltoo}{\mb{L^2({\mathbf H})}} \nc{\rtoo}{\mb{{\mathbf R}^2}}
\nc{\slr}{{\mathbf {SL}}(2,\real)} \nc{\slz}{{\mathbf {SL}}(2,\ints)}
\nc{\su}{{\mathbf {SU}}(1,1)} \nc{\so}{{\mathbf {SO}}}
\nc{\hyp}{{\mathbb H}}
\nc{\disc}{{\mathbf D}}
\nc{\torus}{{\mathbb T}}
\newcommand{\tk}{\widetilde{K}}
\newcommand{\boe}{{\bf e}}\newcommand{\bt}{{\bf t}}
\newcommand{\vth}{\vartheta}
\newcommand{\CGh}{\widetilde{\CG}}
\newcommand{\db}{\overline{\partial}}
\newcommand{\tE}{\widetilde{E}}
\newcommand{\tr}{\mbox{tr}}
\newcommand{\ta}{\widetilde{\alpha}}
\newcommand{\tb}{\widetilde{\beta}}
\newcommand{\txi}{\widetilde{\xi}}
\newcommand{\hV}{\hat{V}}
\newcommand{\IC}{\mathbf{C}}
\newcommand{\IZ}{\mathbf{Z}}
\newcommand{\IP}{\mathbf{P}}
\newcommand{\IR}{\mathbf{R}}
\newcommand{\IH}{\mathbf{H}}
\newcommand{\IG}{\mathbf{G}}
\newcommand{\CC}{{\mathcal C}}
\newcommand{\CS}{{\mathcal S}}
\newcommand{\CG}{{\mathcal G}}
\newcommand{\CL}{{\mathcal L}}
\newcommand{\CO}{{\mathcal O}}
\nc{\ca}{{\mathcal A}} \nc{\cag}{{{\mathcal A}^\Gamma}}
\nc{\cg}{{\mathcal G}} \nc{\chh}{{\mathcal H}} \nc{\ck}{{\mathcal B}}
\nc{\cl}{{\mathcal L}} \nc{\cm}{{\mathcal M}}
\nc{\cn}{{\mathcal N}} \nc{\cs}{{\mathcal S}} \nc{\cz}{{\mathcal Z}}
\nc{\cM}{{\mathcal M}}
\nc{\sind}{\sigma{\rm -ind}}
\newcommand{\la}{\langle}
\newcommand{\ra}{\rangle}
\newcommand{\ual}{\underline{\al}}
\renewcommand{\labelitemi}{{}}
 \title{A noncommutative Atiyah-Patodi-Singer index theorem in KK-theory}
        \author{A. L. Carey} \author{J. Phillips} \author{A. Rennie}
        \address{Mathematical Sciences Institute, Australian National University, Canberra, ACT, Australia.}
\address{Department of Mathematics and Statistics, University of Victoria,
Victoria, BC, Canada.}
\address{Department of Mathematics, Copenhagen University, Universitetsparken
5, Copenhagen, Denmark.}
\email{ acarey@maths.anu.edu.au,phillips@math.uvic.ca,
        rennie@maths.anu.edu.au}
        \maketitle

\centerline{{\bf Abstract}}

We investigate an extension of ideas of Atiyah-Patodi-Singer
(APS) to a noncommutative geometry setting
framed in terms of Kasparov modules.
We use a mapping cone construction to relate odd index pairings
to even index pairings
with APS boundary conditions in the setting of $KK$-theory,
generalising the commutative theory.
We find that Cuntz-Kreiger systems
provide a natural class of examples for our construction and
 the index pairings coming from APS boundary
conditions yield complete $K$-theoretic information about
certain graph $C^*$-algebras.
\parskip=0.0cm
\tableofcontents
\parskip=0.3cm

\section{Introduction}

This paper is about a noncommutative
analogue of APS index theory. 
We will
focus on one aspect of generalising the APS theory. Namely we replace  classical
first order elliptic operators on a manifold with product metric near the
boundary by a `cylinder' of operators on a Kasparov module. We explain
below how the classical theory provides an example of this more general
framework. We also show in the last Section that there are many noncommutative
examples as well. Our motivation is not simply that we are trying to understand noncommutative manifolds with boundary but is derived from the fact that the construction in this paper
can be applied to many index problems in semifinite noncommutative geometry using \cite{KNR}
(which we plan to address elsewhere).

To explain our point of view
let us recast  a simple special case,
using the language of later Sections,
the connection between spectral flow and APS boundary
conditions discussed in \cite{APS3}. Let $X$
be a closed Riemannian manifold, of odd dimension, and let $\D$
be a (self-adjoint) Dirac type operator on $X$. Then $\D$ determines an odd
$K$-homology class $[\D]$ for the algebra $C(X)$ and  we may pair
$[\D]$ with the $K$-theory class of a unitary $u\in M_k(C(X))$ to obtain the
integer
$$ \mbox{Index}(P_kuP_k)=sf(\D_k,u\D_k u^*).$$
Here $P_k$ is the nonnegative spectral projection for $\D_k:=\D\otimes
Id_{\C^k}$ and the index
of the `Toeplitz operator' $P_kuP_k$ gives the spectral flow
$sf(\D_k,u\D_ku^*)$ from $\D_k$ to
$u\D_k u^*$.

We may also attach a semi-infinite cylinder to $X$, and consider
the manifold-with-boundary
$X\times\R_+$. If $\D$ acts on sections of some bundle
$S\to X$, then $\D$ determines a self-adjoint operator on the $L^2$-sections of
$S$,
$\HH=L^2(X,S)$, with respect to an appropriate measure constructed from
the Riemannian metric and bundle inner products. We define
$$\hat\HH=\bca L^2(\R_+,\HH)\\
L^2(\R_+,\HH) \oplus\Phi_0\HH\eca
,\ \ \hat\D=\bma 0 & -\p_t+\D\\ \p_t+\D&0\ema,$$
where $\Phi_0$ is the projection onto the kernel of $\D$. It is
necessary to single out the zero eigenvalue of $\D$ for special
attention since it gives rise to `extended $L^2$-solutions' which
contribute to the index, \cite{APS1}. We let $\hat\D$ act as zero on
 $\Phi_0\HH$, and regard this subspace as being
composed of values at infinity of extended solutions (more on this in
the text).

We give $\hat\D$ APS boundary conditions. That is, we take the domain
of $\p_t+\D$ to be
$$ \{\xi\in L^2(\R_+,\HH):(\p_t+\D)\xi\in L^2(\R_+,\HH),\
P\xi(0)=0\}$$
where again $P$ is the nonnegative spectral projection for $\D$. The
domain of $-\p_t+\D$ is defined similarly using $1-P$ in place of
$P$. Then it can be shown, see for instance \cite{APS1}, that
$\hat\D$ is an unbounded self-adjoint operator and for any $f\in
C^\infty(X\times\R_+)$ which is of compact support and equal to a
constant on the boundary, the product
$f(1+\hat\D^2)^{-1/2}$
is a compact operator on $\hat\HH$.

Such functions lie in the mapping cone algebra for the inclusion
$\C\hookrightarrow C(X)$. This is defined as
$$M(\C,C(X))=\{f:\R_+\to C(X):f(0)\in\C 1_X,\ \ f\
\mbox{continuous and vanishes at }\infty\}.$$
We have an exact sequence
$$0\to C(X)\otimes C_0((0,\infty))\to M(\C,C(X))\to \C\to 0$$
from which we get a six term sequence in $K$-theory. Since
$K_1(\C)=0$, this sequence simplifies to
$$0\to K_1(C(X))\to K_0(M(\C,C(X))\to K_0(\C)\to K_0(C(X))\to
K_1(M(\C,C(X)))\to 0.$$ A careful analysis, which we present in
greater generality in this paper, shows that the map $\Z=K_0(\C)\to
K_0(C(X))$ takes $n$ to the class of the trivial bundle of rank $n$
on $X$, and so is injective. Thus we find that
$$K_1(C(X))\cong K_0(M(\C,C(X))),$$
and the mapping cone algebra is providing a suspension of sorts. The
relationship between the even index pairing for $\hat\D$ and the odd
index pairing for $\D$ is then as follows. Let $e_u$ be the
projection over $M(\C,C(X))$ determined by the unitary $u$ over
$C(X)$, so that $[e_u]-[1]\in K_0(M(\C,C(X)))$. Then
$$\mbox{Index}(e_u(\p_t+\D)e_u)-\mbox{Index}(\p_t+\D)=
\la[e_u]-[1],[\hat\D]\ra=\la[u],[\D]\ra
=sf(\D,u\D u^*).$$

The purpose of this paper is to present a noncommutative
analogue of this picture.  Our main result, Theorem \ref{mainresult},
shows that the situation described above for the commutative case
carries over to a class of Kasparov modules for noncommutative algebras.
We exploit a 
paper of Putnam \cite{Put}
on the K-theory of mapping cone algebras to give  
an  APS type construction for a Kasparov module with
boundary conditions that implies
an equality between even and odd indices.
Not only will we find a new version of this index
equality, but we will see that it allows us to use APS boundary
conditions to obtain interesting index pairings, and consequences, 
that were
previously unknown. For instance we show that the 
complicated $K$-theory 
calculations of \cite{PR} can be given a simple functorial description.


A description of the organisation and main results of the paper now follows. 
We begin  in the next Section with some preliminaries on
Kasparov modules. 
 In Section \ref{map} we review
\cite{Put}, describing $K_0$ of
mapping cone algebras, $M(F,A)$
where $F\subset A$ are certain $C^*$-algebras (replacing the pair
$\C\subset C(X)$ in the classical setting above). We make some
basic computations related to these groups and associated exact
sequences.

The application of APS boundary conditions for Kasparov modules is
done in Section \ref{secAPS}. 
We show that certain odd Kasparov modules for
algebras $A,\,B$ with $F$ a subalgebra of $A$, can be
`suspended'  to obtain even Kasparov modules for the algebras $M(F,A),\,B$,
using APS boundary conditions. The proof is surprisingly complicated
as there are substantial technical issues. Even
self-adjointness of the abstract Dirac operator on the suspension
with APS boundary conditions is not clear.  We solve all of the difficulties using a careful
 construction in the noncommutative setting of a parametrix for our abstract Dirac operators on the even Kasparov module.

The main theorem (Theorem 5.1) shows that two index pairings -- one
from an odd Kasparov module  and one from its even `suspension'  -- with values in $K_0(B)$ are equal. Replacing
$K_0(\C)=\Z$ with $K_0(B)$ gives us an analogue of the
classical example above. The proof is quite difficult; solving differential 
equations in Hilbert $C^*$-modules is a more complex issue than in Hilbert space.

In Section 6 we explain one class of 
examples. There we calculate the $K$-groups of the
mapping cone algebra $M(F,A)$ 
for the inclusion of the fixed point algebra $F$
of the gauge action on certain graph $C^*$-algebras $A$. 
For these algebras, the application
of Theorem 5.1 yields in Proposition 5.7
an isomorphism from $K_0(M(F,A))$ to
$K_0(F)$, which leads to a functorial description
of the calculations of $K_0(A),\,K_1(A)$ in \cite{PR}.
 
 Readers familiar with \cite{B-B} may be puzzled by the fact that we do not
study the more general question of boundary conditions parametrised by a
Grassmanian.
In fact we make, in our main theorem, an assumption that classicially
corresponds to assuming that we can work with a fixed APS boundary
condition for all of the perturbed operators we study. We know that for
classical index problems it is often the case that a more general operator
can be homotopied to one that preserves the APS boundary conditions. In
the noncommutative context of this paper we have not studied this homotopy
argument. The examples in Section 6 illustrate that for many cases our
restricted analysis suffices and provides complete information about the
$K$-theory of the relevant algebras.

{\bf Acknowledgements}. We thank Rsyzard Nest for advice on Section 5, David Pask, Aidan
Sims and  Iain Raeburn for enlightening conversations and Ian Putnam for bringing his work to the third author's attention.
The first and second named authors
acknowledge the financial assistance of the Australian Research
Council and the Natural 
Sciences and Engineering Research Council of Canada
while the third named author thanks  Statens
Naturvidenskabelige Forskningsr{\aa}d, Denmark. All authors are grateful
for the support of the Banff International Research Station where 
some of this
research was undertaken.

\section{Kasparov modules}\label{graph}


The Kasparov modules considered in this subsection are for 
$C^*$-algebras
with  trivial grading.

\begin{defn}\label{Kasparov} An {\bf odd Kasparov $A$-$B$-module} consists of a
countably generated ungraded right $B$-$C^*$-module $E$, with 
$\phi:A\to
End_B(E)$ a $*$-homomorphism, together with $P\in End_B(E)$  such that
$a(P-P^*),\ a(P^2-P),\ [P,a]$ are all compact
endomorphisms. Alternatively, for $V=2P-1$, 
$a(V-V^*),\ a(V^2-1),\ [V,a]$ are all compact endomorphisms 
for
all $a\in A$. One can modify $P$ to $\tilde P$ so that $\tilde P$ is 
self-adjoint; $\n\tilde P\n\leq 1$; $a(P-\tilde P)$ is compact for 
all $a\in A$ 
and the other conditions for $P$ hold with $\tilde P$ in place of $P$
without changing the module $E$. If $P$ has a spectral gap about $0$
(as happens in the cases of interest here) then we may and do assume that 
$\tilde P$
is in fact a projection without changing the module, $E$. 
(Note 
that by
17.6 of \cite{Bl} we may assume that $P$ is a projection by 
changing to a new module in the same class as $E$. )
\end{defn}

By \cite{K}, [Lemma 2, Section 7], the pair $(\phi,P)$ determines a
$KK^1(A,B)$ class, and every class has such a representative. The
equivalence relation on pairs $(\phi,P)$ that give $KK^1$ classes
is generated by unitary equivalence $(\phi,P)\sim (U\phi U^*,UPU^*)$ 
and
homology: $(\phi_1,P_1)\sim (\phi_2,P_2)$ if 
$P_1\phi_1(a)-P_2\phi_2(a)$ is a
compact endomorphism for all $a\in A$, see also \cite[Section 7]{K}.
Later we will also require {\bf even}, or {\bf graded},  Kasparov modules.
\begin{defn}\label{even} 
An {\bf even Kasparov $A$-$B$-module} has, in addition to the data of the previous definition,
a grading by a self-adjoint
endomorphism
$\Gamma$ with $\Gamma^2=1$ and
$\phi(a)\Gamma=\Gamma\phi(a)$, $V\Gamma+\Gamma V=0$.
\end{defn}

The next theorem presents a general result used in \cite{pr}[Appendix] 
about the Kasparov product in the odd case.

\begin{thm} Let $(Y,T)$ be an odd Kasparov module for the
$C^*$-algebras $A,B$. Then (assuming that $T$ has a spectral gap 
around $0$)
the Kasparov product of $K_1(A)$ with the
class of $(Y,T)$ is represented by
$$\la [u],[(Y,T)]\ra= [\ker PuP]-[{\rm coker} PuP]\in K_0(B),$$
where $P$ is the non-negative spectral projection for the self-adjoint 
operator $T$.
\end{thm}

This pairing was studied in \cite{pr}, as well as the relation to
the semifinite local index formula in noncommutative geometry. 
It is also the starting point for this
work. 
More detailed information about the 
$KK$-theory version  of this can be found in \cite{KNR}.

In this paper we will employ unbounded representatives of 
$KK$-classes. 
The theory of unbounded operators on $C^*$-modules that we require
is all contained in Lance's book, \cite{L}, [Chapters 9,10]. We quote
the following definitions (adapted to our situation).

\begin{defn} Let $Y$ be a right $C^*$-$B$-module. A densely defined
unbounded operator $\D:{\rm dom}\ \D\subset Y\to Y$ is a $B$-linear
operator defined on a dense $B$-submodule ${\rm dom}\ \D\subset Y$.
The operator $\D$ is {\bf closed} if the graph
$ G(\D)=\{(x, \D x):x\in{\rm dom}\ \D\}$
is a closed submodule of $Y\oplus Y$.
\end{defn}

If $\D:\mbox{dom}\ \D\subset Y\to Y$ is densely defined and
unbounded, we define the domain of the {\bf adjoint} of $\D$ to be
the submodule:
$$\mbox{dom}\ \D^*:=\{y\in Y:\exists z\in Y\ \mbox{such that}\
 \forall x\in\mbox{dom}\ \D, \la\D x|y\ra_R=\la x|z\ra_R\}.$$
Then for $y\in \mbox{dom}\ \D^*$ define $\D^*y=z$. Given
$y\in\mbox{dom}\ \D^*$, the element $z$ is unique, so
$\D^*:\mbox{dom}\D^*\to Y$, $\D^*y=z$ is well-defined, and moreover
is closed.

\begin{defn}
Let $Y$ be a right $C^*$-$B$-module. A densely defined unbounded
operator $\D:{\rm dom}\ \D\subset Y\to Y$ is {\bf symmetric} if for all
$x,y\in{\rm dom}\ \D$
$$ \la\D x|y\ra_R=\la x|\D y\ra_R.$$
A symmetric operator $\D$ is {\bf self-adjoint} if ${\rm dom}\ \D={\rm
dom}\ \D^*$ (so $\D$ is closed). A densely defined
operator $\D$ is {\bf regular} if $\D$ is closed, $\D^*$ is
densely defined, and $(1+\D^*\D)$ has dense range.
\end{defn}
The extra requirement of regularity is necessary in the $C^*$-module
context for the continuous functional calculus, and is not
automatic, \cite{L},[Chapter 9].

\begin{defn}\label{unbddKasparov} An {\bf odd 
unbounded Kasparov $A$-$B$-module} consists of a
countably generated ungraded right $B$-$C^*$-module $E$, with 
$\phi:A\to
End_B(E)$ a $*$-homomorphism, together with an unbounded
self-adjoint regular operator $\D:{\rm dom}\D\subset E\to E$ such that
$[\D,a]$ is bounded for all $a$ in a dense $*$-subalgebra of $A$ and
$a(1+\D^2)^{-1/2}$ is a compact endomorphism of $E$ for all 
$a\in A$. An {\bf even unbounded Kasparov $A$-$B$-module} 
has, in addition to the previous data, 
 a $\Z_2$-grading
with $A$ even and $\D$ odd, as in Definition \ref{even}.
\end{defn}

\section{$K$-Theory of the Mapping Cone Algebra and pairing with
$KK$-theory}\label{map}

\subsection{The mapping cone}

Let $F\subset A$ be a $C^*$-subalgebra of a $C^*$-algebra $A$.
Recall \cite{Put} that the mapping cone algebra is
$$ M(F,A)=\{f:[0,1]\to A: f\ \mbox{is continuous},\ f(0)=0,\
f(1)\in F\}.$$ The algebra operations are pointwise
addition and multiplication and the norm is the
uniform (sup) norm. There is a natural exact sequence
$$ 0 \to C_0(0,1)\otimes A\sr{i}{\to} M(F,A)\sr{ev}{\to} F\to 0.$$
Here $ev(f)=f(1)$ and $i(g\otimes a)=t\to g(t)a$.  It is well known
that when $F$ is an ideal in the algebra $A$ we have
$K_*(M(F,A))\cong K_*(A/F)$.

We will always be considering the situation where $K_1(F)=0$, as is
the case for graph $C^*$-algebras, though this is not strictly necessary.
When $K_1(F)=0$, the six term sequence in $K$-theory coming
from this short exact sequence degenerates into \be 0\to K_1(A)\to
K_0(M(F,A))\sr{ev_*}{\to}K_0(F)\sr{j_*}{\to} K_0(A)\to
K_1(M(F,A))\to 0.\label{sixterm}\ee
We need to justify the notation $j_*$; namely we need to display the
map $j$ which induces $j_*$.

\begin{lemma}
\label{lm:j}
In the above exact sequence the map $j_*:K_0(F)\to
K_0(A)$ is induced by minus the inclusion map $j:F\to A$
(up to Bott periodicity).
\end{lemma}

\begin{proof} The map we have denoted by $j_*$ is actually a
composite:
$$j_*:K_0(F)\sr{\p}{\to}K_1(C_0(0,1)\otimes A)\sr{\cong}{\to}K_0(A).$$
The isomorphism here is the inverse of the Bott map $Bott:K_0(A)\to
K_1(C_0(0,1)\otimes A)$, where
$Bott([p])=[e^{-2\pi i t}\otimes p+1\otimes(1-p)]$.
The boundary map $\p$ is defined as follows, \cite[p 113]{HR}. For
$[p]-[q]\in K_0(F)$, we choose representatives $p,q$ over $F$, and
then choose self-adjoint lifts $x,y$ over $M(F,A)$. Then
$e^{2\pi ix},e^{2\pi iy}$ are unitaries over $C(S^1)\otimes A$ which
 are equal to the identity modulo $C_0(0,1)\otimes A$. Then
$$ \p([p]-[q])=[e^{2\pi ix}]-[e^{2\pi iy}]\in K_1(C_0(0,1)\otimes A).$$
Now we choose the particular lifts over $M(F,A)$ given by $x(t)=tp$ and
$y(t)=tq$ (in fact these are $t\otimes j(p)$ and
$t\otimes j(q)$). Both these elements are self-adjoint, vanish at
$t=0$ and at $t=1$ are in $F$. Now
$$\left[e^{2\pi ix}\right]-\left[e^{2\pi iy}\right]=
\left[e^{2\pi it\otimes p}\right]-\left[e^{2\pi it\otimes q}\right]=
-Bott([p]-[q])\in K_1(C_0(0,1)\otimes A).$$
So modulo the isomorphism $Bott:K_0(A)\to K_1(C_0(0,1)\otimes A)$,
$j_*([p]-[q])=-([j(p)]-[j(q)]).$
\end{proof}

We now describe $K_0(M(F,A))$ \cite{Put}. Let
$V_m(F,A)$ be the set of partial isometries $v\in M_m(A)$ such
that $v^*v,\ vv^*\in M_m(F)$. Using the inclusion
$V_m\hookrightarrow V_{m+1}$ given by $v\to v\oplus 0$ we can
define
$$ V(F,A)=\cup_mV_m(F,A).$$
Our aim, following \cite{Put}, is to define a map
$\kappa:V(F,A)\to K_0(M(F,A))$, and
we proceed in steps. First, let $v\in V(F,A)$ and define a 
self-adjoint unitary
$v_1$ via:
$$v_1=\bma 1-vv^* & v\\ v^*& 1-v^*v\ema,$$
that is, $v_1^2=1,\ v_1=v_1^*$. So, $v_1=p_+ - p_{-}$ where
$p_+=\frac{1}{2}(v_1 +1)$ and $p_{-}=\frac{1}{2}(1-v_1)$ are the
positive and negative spectral projections for $v_1.$ Then for $t\in[0,1]$
define
$$ v_2(t)=p_+ +e^{i\pi t}p_{-}$$
so that we have a continuous path of
unitaries from the identity ($t=0$) to $v_1$ ($t=1$).
Observe that $v_2(t)$ is unitary for all $t\in[0,1]$,
$v_2\in C([0,1])\otimes M_{2m}(A)$, $v_2(0)=1$ and $v_2(1)=v_1$.
Now define
$$ e_v(t)=v_2(t)ev_2(t)^*,\ \ \ e=\bma 1 & 0\\0 & 0\ema.$$
Then $e_v(t)$ is a projection over the unitization $\tilde M(F,A)$ of
$M(F,A)$ given by
$$ \tilde M(F,A)=\{f:[0,1]\to \tilde A:f\ \mbox{is continuous},\
f(0)\in \C1,\
f(1)\in \tilde F\}.$$ Thus $[e_v]-[e]$ defines an element of
$K_0(M(F,A))$. So with
$ \kappa(v)=[e_v]-[e]$ we find:

\begin{lemma}\cite[Lemmas 2.2,2.4,2.5]{Put} \label{Im:Put}\\
1) $\kappa(v\oplus w)=\kappa(v)+\kappa(w)$\vspace{-4pt}

2) If $v,w\in V_m(F,A)$ and $\Vert v-w\Vert<(200)^{-1}$ then
$\kappa(v)=\kappa(w)$\vspace{-4pt}

3) If $v\in V_m(F,A)$, $w_1,w_2\in U_m(F)$ then $w_1vw_2\in
V_m(F,A)$, $\kappa(w_1)=\kappa(w_2)=0$,
$\kappa(w_1vw_2)=\kappa(v)$.\vspace{-4pt}

4) For $v\in M_m(F)$ a partial isometry, $\kappa(v)=0$, so,
for $p\in M_m(F)$ a projection, $\kappa(p)=0$. \vspace{-4pt}

5) The map $\kappa:V(F,A)\to
K_0(M(F,A))$ is onto.\vspace{-4pt}

6) Generate an equivalence relation
$\sim$ on $V(F,A)$ by\vspace{-5pt}

(i) $v\sim v\oplus p$ for $v\in V(F,A)$, $p\in M_k(F)$\vspace{-4pt}

(ii) If $v(t),\ t\in[0,1]$ is a continuous path in $V(F,A)$ then
$v(0)\sim v(1)$.\vspace{-4pt}

Then $\kappa:V(F,A)/\sim\to K_0(M(F,A))$ is a well-defined bijection.
\end{lemma}\vspace{-3pt}
Hence we may realise $K_0(M(F,A))$ as equivalence classes of partial
isometries in $M_m(A)$ whose source and range projections lie 
in $M_m(F)$.
Observe that when $K_1(F)=0$, $K_1(A)$ embeds in
 $K_0(M(F,A))$  by regarding a unitary (possibly in a
unitization of $A$) as a partial isometry.
We add the following lemmas which we will need later.
\begin{lemma}\label{composeisoms} Let $v,w\in V_m(F,A)$ have
the same source projection, so
$v^*v=w^*w=p$, say. Then
$[v\oplus w^*]=[v]+[w^*]=[v]-[w]=[vw^*].$
\end{lemma}

{\bf Remark} If $v=p$ we get a proof that $-[w]=[w^*]$.

\begin{proof} The homotopy is given by
$$V_\theta=\bma \cos^2(\theta)v+\sin^2(\theta)p &
\cos(\theta)\sin(\theta)(w^*-vw^*)\\
\cos(\theta)\sin(\theta)(p-v) &
\cos^2(\theta)w^*+\sin^2(\theta)vw^*\ema,
\ \ \ \ \theta\in[0,\pi/2].$$
\end{proof}

\begin{lemma}\label{orthogsum} Suppose $v^*v=p+q$ with $p,q\in F$ 
projections, $p\perp
q$. Then $v=vp+vq$, $vv^*=vpv^*+vqv^*$, $vpv^*\perp vqv^*$ and if we 
assume that
$vpv^*\in F$ then
$[v]=[vp\oplus vq]=[vp]+[vq].$
\end{lemma}

\begin{proof} The first few statements are simple algebraic consequences
of the hypothesis. The homotopy from $v\sim v\oplus0$ to 
$vp\oplus vq$ is
 $$V_\theta=\bma
vp+vq\cos^2(\theta)&vq\sin(\theta)\cos(\theta)\\vq\sin(\theta)
\cos(\theta)
& vq\sin^2(\theta)\ema,\ \ \ \ \ \theta\in[0,\pi/2].$$
\end{proof}

We will use the following equivalent definition of the
mapping cone algebra, as it is more useful for our intended
applications and agrees with the definition in the classical
commutative case. We let
$$ M(F,A)=\{f:\R_+\to A: f\ \mbox{continuous and 
vanishes at}\ \infty\
\mbox{and}\ f(0)\in F\}.$$
This way of defining the mapping cone algebra
gives an isomorphic $C^*$-algebra
and we will take this as our definition from now on.

\subsection{The pairing in $KK$ for the mapping cone}

Using the Kasparov product, $K_0(M(F,A))$ pairs with
$KK^0(M(F,A),B)$ for any $C^*$-algebra $B$. However, $K_0(M(F,A))$
also pairs with odd $A,B$ Kasparov modules $(Y,V)$ such that the
left action by $f\in F\subset A$ commutes with $V$.
While all our constructions work for such $A,B$ Kasparov modules, we
will restrict in the sequel to $A,F$ Kasparov modules. This will
cause no loss of generality to those wishing to extend these results
to the general case, but is the situation which arises naturally 
in examples.

{\bf Standing Assumptions (SA).} For the rest of this Section, let $v\in
A$ be an isometry with $v^*v,vv^*\in F$ (the same will work for
matrix algebras over $A$, $F$). Let $(Y,V)$ be an odd Kasparov
module for $A,F$ such that the left action of $f\in F\subset A$
commutes with $V =2P-1$ 
where $P$ is the non-negative spectral 
projection
 of $V$. 

To define the pairing we need a preliminary result.

\begin{lemma} Let $(Y,V)$ satisfy {\bf SA}. 
The two projections $ vv^*P$ and $vPv^*$
differ by a compact endomorphism, and consequently  $PvP:v^*vP(Y)\to vv^*P(Y)$
is Fredholm.\end{lemma}

\begin{proof}
It is a straightforward calculation that
$$vPv^*=vv^*P+v[P,v^*]=vv^*P+\frac{1}{2}v[V,v^*]$$ and, as $[V,v^*]$
is compact, $vv^*P$ and $vPv^*$ differ by a compact endomorphism.
One easily checks that $Pv^*P:vv^*P(Y)\to v^*vP(Y)$ is a parametrix for
$PvP$ and the second statement follows.
\end{proof}

\begin{rems*}
In all the calculations we do here, if $v\in M_k(A)$ then we use
$P_k:=P\otimes 1_k$ in place of $P$: we will {\bf usually} suppress this
inflation notation  in the interests of avoiding notation inflation.
\end{rems*}

\begin{defn}  For $[v]\in K_0(M(F,A))$ and $(Y,2P-1)$ 
satisfying {\bf SA}, define
$$ [v]\times(Y,V)={\rm Index}(PvP:v^*vP(Y)\to vv^*P(Y))=[\ker PvP]-[{\rm
coker} PvP]
\in K_0(F).$$
\end{defn}

We make some general observations.
 
$\bullet$ If $v$ is unitary over $A$, we recover the usual Kasparov 
pairing
between $K_1(A)$ and $KK^1(A,F)$, \cite{KNR}, 
\cite[Appendix]{pr}. Thus
the pairing depends only on the class of $(Y,2P-1)$ in
$KK^1(A,F)$ for $v$ unitary.

$\bullet$ In general the operator $PvP$ does not have closed
range. However the operator
$$ \widetilde{PvP}:=\bma PvP & 0\\ (1-P)vP & 0\ema:\bca v^*vP(Y)
\\ v^*vP(Y)\eca\to \bca
vv^*P(Y)\\ vv^*(1-P)(Y)\eca$$
does have closed range, \cite[Lemma 4.10]{GVF}, and the index is
easily seen to be
$$\mbox{Index}(\widetilde{PvP})=\left[\bca Pv^*(1-P)(Y)\\
(1-P)v^*P(Y)\eca\right] - \left[\bca (1-P)v^*P(Y)\\
(1-P)v^*P(Y)\eca\right].$$ The index of $PvP$ is in fact defined to be
the index of any suitable `amplification' like $\widetilde{PvP}$,
\cite{GVF}, and we see that if the right $F$-module
$Pv^*(1-P)(Y)$ is closed, then the `correction' term $(1-P)v^*P(Y)$
arising from the amplification process cancels out. Since the
$K$-theory class of the index does not in fact depend on the choice
of amplification, we will ignore this subtlety from here on. That
is, we assume without any loss of generality that the various
Fredholm operators we consider satisfy the stronger condition of being
{\bf regular} in the sense of having
a pseudoinverse \cite{GVF}[Definition 4.3]. 
Since we will be
concerned only with showing that certain indices coincide, this will
not affect our conclusions.

$\bullet$ The pairing  depends only on the class of $v$ in
$K_0(M(F,A))$ with the module $(Y,V)$ held fixed, 
in particular it vanishes if $v\in F$. These
statements follow in the same way as the analogous statements for
unitaries, cf \cite[Appendix]{pr}.

$\bullet$ Since addition in the ``Putnam picture'' of $K_0(M(F,A))$
is by direct sum as is addition in the usual picture of $K_0(A)$
it is easy to see that the pairing is additive in the $K_0(M(F,A))$
variable with the module $(Y,V)$ held fixed. So with $(Y,V)$
held fixed we have a well-defined group homomorphism:
$$\times(Y,2P-1) : K_0(M(F,A))\to K_0(A).$$

\subsection{Dependence of the pairing on the choice of $(Y,2P-1)$}

The dependence on the Kasparov module $(Y,2P-1)$ is not
straightforward. For instance, we require that $P$ commute with the
left action of $F$, and so homotopy invariance is necessarily
broken.
%
%
%
We now fix $v\in V_m(F,A)$ and show that we can obtain an {\bf even}
Kasparov module $(Y_v,R_v)$ for $(A_v,F):=(vv^*Avv^*,F)$ so that the 
two classes 
$[v]\times(Y,2P-1)$ and $[1_{A_v}]\times[(Y_v,R_v)]$ are equal 
in $K_0(A)$, with the latter being a Kasparov product of {\em genuine} 
$KK$-classes.

The purpose in doing this is to understand the 
homotopy invariance properties of 
$\mbox{Index}(PvP)$ by characterising it as a 
Kasparov product. In this subsection this is 
achieved by creating a 
`smaller' Kasparov module, which depends on $v$. 
In our main theorem, Theorem 5.1, we associate to 
an odd unbounded Kasparov module $(X,\D)$ a 
`larger' even unbounded Kasparov module 
$(\hat{X},\hat\D)$. This latter module is independent 
of $v$ and allows us to characterise, for all 
$[v]\in K_0(M(F,A))$, the class $\mbox{Index}(PvP)$ 
as  the Kasparov product 
$[v]\times[(\hat{X},\hat\D)]$. 

\begin{lemma} With $v,(Y,2P-1)$ as above, the pair
$$ (Y_v,R_v):=\left(\bca vv^*(Y)\\v^*v(Y)\eca,\bma 0 & R_{-}\\
R_{+} & 0 \ema\right)\;\;where\;\;R_{-}=(PvP-(1-P)v)\;\;and\;\;R_{+}=R_{-}^*$$
is an even $(vv^*Avv^*,F)$ Kasparov module for the representation
 $$ \pi(a)=\bma a & 0\\ 0 & v^*av\ema\; for\;a\in vv^*Avv^*.$$
\end{lemma}

\begin{proof} First observe that $vv^*Avv^*$ is always unital, with 
unit $1_{A_v}=vv^*$, and that $\pi(a)$ leaves $Y_v$ invariant for $a\in vv^*Avv^*.$ 
Next, $R_v$ is clearly self-adjoint and moreover, 
$R_{-}v^*v=vv^*R_{-}$. Taking adjoints we obtain $R_{+}vv^*=v^*vR_{+}$
so that $R_v$ also leaves $Y_v$ invariant. Now since $v$ and
 $v^*$ commute with $P$ up to compacts we see that 
\begin{eqnarray}\label{Rminus}R_{-}=(2P-1)v\;(mod \;compacts)&=&v(2P-1)\;(mod\;compacts)
\;\;and\\
\label{Rplus}R_{+}=(2P-1)v^*\;(mod \;compacts)&=&v^*(2P-1)\;(mod\;compacts).
\end{eqnarray}
Hence, 
$$R_v^2= \bma vv^* & 0\\
          0 & v^*v \ema =1_{Y_v}\;\;(mod \;compacts).$$
The compactness of commutators $[R_v,\pi(a)]$ can be reduced
by (\ref{Rminus}) and (\ref{Rplus}) to the equations:
$$a(2P-1)v =(2P-1)vv^*av\;\;\mbox{and}\;\;v^*avv^*(2P-1)=v^*(2P-1)a\;
(mod\;compacts).$$
This completes the proof using $a=vv^*a=avv^*$ and $[P,a]$  compact.
\end{proof}

The following corollary is obvious once we note that
$$ \pi(1_{A_v})=\bma vv^* & 0\\ 0 & v^*v\ema$$
\begin{cor} We have the equality in $K_0(F)$:
$[v] \times (Y,2P-1)=[1_{A_v}] \times [(Y_v,R_v)].$
Hence the pairing $[v]\times (Y,2P-1)$ depends only on $[v]\in
K_0(M(F,A))$ and the class $[(Y_v,R_v)]\in KK^0(vv^*Avv^*,F)$.
\end{cor}

{\bf Remarks}.  In the Kasparov module $(Y_v,R_v)$ there is
a dependence on $v$. 
This result also shows that we can pair with any subprojection
of $vv^*$ in $F$ instead of $vv^*=1_{vv^*Avv^*}$.
The Kasparov module $(Y_v,R_v)$ is formally reminiscent
of the module obtained by a cap product of an odd module with a
unitary.
The remaining homotopy invariance is for homotopies of
operators on $Y_v$, or operators on $Y$ commuting with $vv^*$.




It should be clear by now that the mapping cone algebra
provides a partial suspension, but mixes odd and even in a fascinating way.
In the next section we relate the even index pairing for $M(F,A)$ to
the odd index pairing described here.

\section{APS Boundary Conditions and Kasparov Modules for the Mapping Cone}
\label{secAPS}

In this Section we begin the substantially new material
by constructing an even Kasparov module for the mapping cone
algebra $M(F,A)$
starting from an odd Kasparov $F$-module $(X,\D)$ for $A.$ In particular we are
assuming that $\D$ is self-adjoint and regular on $X$, has discrete spectrum
and the eigenspaces are closed $F$-submodules of $X$ which sum to $X$.
Our even module $\hat{X}$ is initially defined to be the direct sum of two copies
of the $C^*$-module:
$\E=L^2(\R_+)\otimes_\C X$
which is the completion of the algebraic tensor
product in the tensor product $C^*$-module  norm. That is, we take finite sums of elementary tensors which can naturally be regarded as
functions $f:\R_+\to X$.  The inner product on such $f=\sum_i f_i\otimes
x_i$, $g=\sum_j g_j\otimes y_j$ is defined to be
$$ \la f|g\ra_\E=\sum_{i,j}\int_0^\infty \bar{f_i}(t)g_j(t)dt \,\la x_i|y_j\ra_X,$$
where we have written $\la\cdot|\cdot\ra_X$ for the inner product on $X$. Clearly 
the collection of all continuous compactly supported functions from $\R_+$
to $X$ is naturally contained in the completion of this algebraic 
tensor product and for such functions $f,g$ the inner product is given by:
$$\la f|g\ra_\E=\int_0^\infty \la f(t)|g(t)\ra_Xdt.$$
The corresponding norm is
$$||f||_\E=||\la f|f\ra_\E||^{1/2}.$$

\begin{rems*}  While many elements in the completion $\E$ can be realised as 
functions it may {\bf not} be true that all of $\E$
consists of $X$-valued functions. We also note that the Banach space
$L^2(\R_+,X)$ of functions $f$ defined by square-integrability of 
$t\mapsto \Vert f(t)\Vert$ is {\bf strictly} contained in $\E.$ 
However, we shall show below that the domain of the operator $\p_t\otimes 1$
on $\E$ (free boundary conditions) consists of $X$-valued functions
which are square-integrable in the $C^*$-module sense above.
We will define our operators using APS boundary conditions on the
domains. 
\end{rems*}
\subsection{Domains}\label{dom}
Let $P$ be the spectral projection for $\D$ corresponding to the 
nonnegative axis and let 
$T_\pm=\pm\p_t\otimes 1+ 1\otimes \D$ ($=\pm\p_t+\D$ for brevity)
with initial domain given by
$$\mbox{dom}\,T_\pm =\{f:\R_+\to X_\D:\; f=\sum_{i=1}^n f_i\otimes x_i,\; f\ 
\mbox{is smooth and compactly supported,}$$ 
$$x_i\in X_{\D},\;P(f(0))=0\ (+\ \mbox{case}),\ 
(1-P)(f(0))=0\ (-\ \mbox{case})\}.$$
 
By smooth we mean $C^\infty$, using one-sided derivatives at
$0\in\R_+$. Then $T_{\pm}:\mbox{dom}\,T_\pm\subset\E\to\E$.
These are both densely defined, and so the operator
$$\hat\D=\bma 0 & T_-\\ T_+ & 0\ema$$
is densely defined on $\E\oplus\E$.
An integration by parts (using the boundary conditions) shows that
$$ (T_\pm f|g)_\E=(f|T_\mp g)_\E,\ \ \ f\in\mbox{dom}\,T_\pm,\
g\in\mbox{dom}\,T_\mp.$$
Hence the adjoints are also densely defined, and so each of these operators is
closable. This shows that $\hat\D$ is likewise closable, and symmetric.

The subtlety noted above, namely that the module $\E$ does not
necessarily consist of functions, forces us to consider some
seemingly circuitous arguments. Basically, to prove
self-adjointness, we require knowledge about domains, and we must
prove various properties of these domains without the benefit of a
function representation of all elements of $\E$. However, we will prove
below a function representation for elements in the natural domain of
$\p_t\otimes 1,$ and therefore in the domains of the closures of $T_{\pm}$
because if  $\{f_j\}\subset\mbox{dom}\,T_\pm$ is a Cauchy sequence in the norm
of $\E$ such that $\{T_\pm f_j\}$ is also Cauchy then as $T_\pm$ is
closable, the limit $f$ of the sequence $f_j$ lies in the domain of
the closure, and $\lim T_\pm f_j=\overline{T_\pm} f$.

\begin{lemma}
\label{lm:first-lemma}
For $f\in {\rm dom}\,T_{\pm}$, the initial domain, we have:\\
(1) $\la T_{\pm}f|T_{\pm}f\ra=\la(\p_t\otimes 1)f|(\p_t\otimes 1)f\ra_\E
+\la(1\otimes\D)f|(1\otimes\D)f\ra_\E \mp \la f(0)|\D(f(0))\ra_X,$ and\\
(2) $\mp \la f(0)|\D(f(0))\ra_X \geq 0.$
\end{lemma}

\begin{proof}
We do the case $T_+$; the proof for $T_-$ is the same.
With a little computation it suffices to see:
$$\la(\p_t\otimes1)f|(1\otimes\D)f\ra_\E +\la(1\otimes\D)f|(\p_t\otimes1)f\ra_\E =
-\la f(0)|\D(f(0))\ra_X$$
for $f=\sum_i f_i\otimes x_i$ with $f_i$ compactly supported and 
$f(0)\in \ker P.$ Then, using integration by parts:
\begin{eqnarray*}
\la(\p_t\otimes1)f|(1\otimes\D)f\ra_\E &=& \sum_{i,j}\int_0^\infty (\frac{d}{dt}
\overline{f_i(t)})(f_j(t)) dt\cdot \la x_i|\D x_j\ra_X\\
&=& -\sum_{i,j}\left\{\overline{f_i(0)}f_j(0+\int_0^\infty
\overline{f_i(t)}\frac{d}{dt}f_j(t)dt\right\} \la x_i|\D x_j\ra_X\\
&=& -\la\sum_i f_i(0)x_i|\sum_j f_j(0)\D x_j\ra_X -
\la\sum_i f_i\otimes x_i|\sum_j\p_t f_j\otimes \D x_j\ra_\E\\
&= &-\la f(0)|\D(f(0)\ra_X-\la f|(\p_t\otimes\D)f\ra_\E.\end{eqnarray*}

But, since $\D$ is self-adjoint and $1\otimes\D$ commutes with $\p_t\otimes 1$
we have $$\la f|(\p_t\otimes\D)f\ra_\E = \la(1\otimes\D)f|(\p_t\otimes 1)f\ra_\E$$
and item (1) follows. To see item (2), we have $(1-P)(f(0))=f(0)$ where
$(1-P)=\mathcal{X}_{(-\infty,0)}(\D)$
so we see that $\D$ restricted to the range of $(1-P)$ is negative and therefore
$-\la f(0)|\D(f(0)\ra_X\geq 0$ in our $C^*$-algebra.
\end{proof}

\begin{cor}
\label{cr:closures}
If $\{f_n\}\subseteq {\rm dom}\,(T_\pm)$ is a Cauchy sequence in the initial domain
of $T_\pm$ and $\{T_\pm(f_n)\}$ is also a Cauchy sequence in $||\cdot||_\E$
norm then both $\{(\p_t\otimes 1)(f_n)\}$ and $\{(1\otimes\D)(f_n)\}$ are
also Cauchy sequences in the $||\cdot||_\E$ norm. Therefore, the limit, $f$ of
$\{f_n\}$ in $\E$  which is in the domain of the
closure of $T_\pm$, is also in the domain of the closures of
both $(\p_t\otimes 1)$ and $(1\otimes\D).$
\end{cor}
\begin{proof}
This follows from the lemma and the fact that if $A=B+C$ are all 
positive elements in a
$C^*$-algebra, then $||A||\geq||B||$ and $||A||\geq||C||$. \end{proof}

\begin{lemma}
\label{lm:got-zero}
{\rm (1)}\;\; If $g=\sum_i f_i\otimes x_i$ where the $f_i$ are smooth and 
compactly supported then
$$\la(\p_t\otimes 1)g|g\ra_\E = -\la g(0)|g(0)\ra_X-\la g|(\p_t\otimes 1)g\ra_\E.$$
{\rm (2)}\;\; With $g$ as above
$$||g(0)||_X^2\leq 2||(\p_t\otimes 1)g||_\E\cdot||g||_\E.$$
\end{lemma}

\begin{proof}
Item (1) is an integration by parts similar to the previous computation
and item (2) follows from item (1) by the triangle and Cauchy-Schwarz
inequalities.
\end{proof}

\subsection{Elements in ${{\rm dom}(\p_t\otimes 1)}$ are functions.}
\begin{defn}
\label{shifts-projections}
For each $t\in\R_+,$ we define two shift operators $S_t$ and $T_t$ on 
$L^2(\R_+)$ via:
$S_t(\xi)(s)=\xi(s+t)$ and $T_t=S_t^*.$ Clearly
both have norm $1$ and $S_tT_t=1$ and $T_tS_t=1-E_t$ where 
$E_t$ is the projection, multiplication by $\mathcal{X}_{[0,t]}.$
Hence, $S_t\otimes 1$, $T_t\otimes 1$, and $E_t\otimes 1$ are in
$\mathcal{L}(\E)$ and $E_t\otimes 1$ converges
strongly to $1_{\E}$ as $t\to\infty.$
\end{defn}

\begin{lemma}
\label{got-functions}
Let $\p_t\otimes 1$ denote the closed operator on $\E$ with free boundary
condition at $0.$ That is, $\p_t\otimes 1$ is the closure of $\p_t\otimes 1$
defined on the initial domain ${\rm dom}^{\prime}(\p_t\otimes 1)$ consisting
of finite sums of elementary tensors $f\otimes x$ where $f$ is smooth and 
compactly supported. Then,\\
\noindent{\rm (1)} $S_t$ leaves ${\rm dom}(\p_t\otimes 1)$ invariant
and commutes with $\p_t\otimes 1.$\\
\noindent{\rm (2)} If $g\in {\rm dom}^{\prime}(\p_t\otimes 1)$ then for each
$t_0\in\R_+$  $$||g(t_0)||_X^2\leq 2||(\p_t\otimes 1)g||_{\E} ||g||_{\E}.$$
\noindent{\rm (3)} If $g\in {\rm dom}(\p_t\otimes 1)$ and $\{g_n\}$ is a sequence
in ${\rm dom}^{\prime}(\p_t\otimes 1)$ with $g_n\to g$ in $\E$ and
$(\p_t\otimes 1)(g_n)\to (\p_t\otimes 1)(g)$ in $\E$ then there is a 
continuous function $\hat{g}:\R_+\to X$ so that $g_n\to \hat{g}$ uniformly
on $\R_+.$ Moreover $\hat{g}\in C_0(\R_+,X)$ and depends only on $g$,
not on the particular
sequence $\{g_n\}.$ \\
\noindent{\rm (4)} If $g\in {\rm dom}(\p_t\otimes 1)$ and $\hat{g}$ is the function
defined in item {\rm (3)} then for all elements $h\in\E$ which are finite sums
of elementary tensors of the form $f\otimes x$ where $f$ is compactly supported 
and piecewise continuous we have:
$$(g|h)_{\E}=\int_0^\infty \la\hat{g}(t)|h(t)\ra_X dt.$$
\noindent{\rm (5)} If $g\in {\rm dom}(\p_t\otimes 1)$ then:
$$\la g|g\ra_{\E}=\lim_{M\to\infty}\int_0^M \la\hat{g}(t)|\hat{g}(t)\ra_X dt :=
\int_0^\infty \la\hat{g}(t)|\hat{g}(t)\ra_X dt.$$
\end{lemma}

\begin{proof}
To see item (1), one easily checks that $S_t\otimes 1$ leaves 
${\rm dom}^{\prime}(\p_t\otimes 1)$ invariant and commutes with $\p_t\otimes 1$
on this space. Since $\p_t\otimes 1$ is the closure of its restriction to
${\rm dom}^{\prime}(\p_t\otimes 1)$ and $S_t\otimes 1$ is bounded the conclusion 
follows by an easy calculation.\\
To see item (2), we apply item (1) and the previous lemma:
\begin{eqnarray*}
||g(t_0)||_X^2&=&||(S_{t_0} g)(0)||_X^2
\leq 2||(\p_t\otimes 1)S_{t_0}(g)||_{\E} ||S_{t_0}(g)||_{\E}\\
&=&2||S_{t_0}(\p_t\otimes 1)(g)||_{\E} ||S_{t_0}(g)||_{\E}\\
&\leq&2||(\p_t\otimes 1)(g)||_{\E} ||g||_{\E}.
\end{eqnarray*} 
To see item (3), apply item (2) to the sequence $\{(g_n -g_m)(t_0)\}$ to see 
that the sequence $\{g_n(t_0)\}$ in $X$ is uniformly Cauchy for $t_0\in\R_+.$
Since we can intertwine two such sequences converging to $g$, we see that
$\hat{g}$ is independent of the particular sequence. That $\hat{g}$ 
vanishes at $\infty$ follows immediately from the uniform convergence.\\
To see item (4), let $\{g_n\}$ be a sequence satisfying the conditions of
item (3). Then for $h$ supported on $[0,M]$ satisfying the conditions of
item (4):
\begin{eqnarray*}
\la g|h\ra_{\E}&=&\lim_{n\to\infty}\la g_n|h\ra_{\E}=\lim_{n\to\infty}\int_0^\infty
\la g_n(t)|h(t)\ra_Xdt\\
&=&\lim_{n\to\infty}\int_0^M\la g_n(t)|h(t)\ra_Xdt=\int_0^M\la\hat{g}(t)|h(t)\ra_Xdt\\
&=&\int_0^\infty\la\hat{g}(t)|h(t)\ra_Xdt.
\end{eqnarray*}
To see item (5), fix $M>0$ and use item (4):
\begin{eqnarray*}
\la g|E_M(g)\ra_{\E}&=&\lim_{n\to\infty} \la g|E_M(g_n)\ra_{\E}
=\lim_{n\to\infty}\int_0^\infty \la\hat{g}(t)|E_M(g_n)(t)\ra_Xdt\\
&=&\lim_{n\to\infty}\int_0^M \la\hat{g}(t)|g_n(t)\ra_Xdt\\
&=&\int_0^M\la\hat{g}(t)|\hat{g}(t)\ra_Xdt.
\end{eqnarray*}
Taking the limit as $M\to\infty$ completes the proof.
\end{proof}

\begin{cor}
\label{got-smooth-functions}
{\rm (1)} If $g\in{\rm dom}(\p_t\otimes 1)^2$ then $(\p_t\otimes 1)(g)$
is also given by a continuous $X$-valued function as above.
{\rm (2)} If $g\in{\rm dom}(\p_t\otimes 1)^n$ for all $n\geq 1$ then
$(\p_t\otimes 1)^n(g)$ is given by a continuous $X$-valued function for 
all $n.$
\end{cor}

\begin{prop}
(1)\;\;If $g\in {\rm dom}\,(\overline{T_\pm})$ the domain of the
closure of $T_\pm$
on its initial domain then $g\in {\rm dom}\,(\overline{\p_t\otimes 1})\cap
{\rm dom}\,(\overline{1\otimes\D})$. Moreover, $g(0)$ is well-defined and
$P(g(0))=0$ in the $T_+$ case while in the $T_-$ case, $(1-P)(g(0))=0$. Furthermore
$$\overline{T_\pm}g=\pm\overline{(\p_t\otimes1)}g + \overline{(1\otimes\D)}g.$$
(2)\;\; If $g\in {\rm dom}\,(\overline{T_\pm})$ as above,
then $g(0)\in {\rm dom}\,(|\D|^{1/2}).$
\end{prop}

\begin{proof}
For the first item, by Corollary \ref{cr:closures},
$g\in \mbox{dom}\,(\overline{\p_t\otimes 1})\cap
\mbox{dom}\,(\overline{1\otimes\D}).$ Then, by the previous Lemma $g(0)$ 
is defined. Since $P$ is a bounded operator
on $X$, $P(g(0))=0$ in the $T_+$ case and $(1-P)(g(0))=0$ in the $T_-$ case.
To see item (2), we use part (2) of Lemma \ref{lm:first-lemma} to see that
for $f\in \mbox{dom}\,(T_\pm)$ we have:
$$\mp \la f(0)\;|\;\D(f(0))\ra_X = \la |\D|^{1/2}(f(0))\;|\; |\D|^{1/2}(f(0))\ra_X.$$
If we apply this observation to $f=g_n-g_m$ where $\{g_n\}$ is a Cauchy
sequence in $\mbox{dom}\,(T_\pm)$ we get the conclusion of item (2).
\end{proof}
\begin{rem*}
Note that evaluation at a point is continuous on ${\rm dom}(\p_t\otimes 1)$
in the ${\rm dom}(\p_t\otimes 1)$-norm, but {\bf not} in the module norm.
\end{rem*}
\subsection{Self-adjointness of $\hat\D$ away from the kernel}

To show that $\hat\D$ is self-adjoint we will follow the basic
strategy of \cite{APS1}
and display a  parametrix which is (almost) an exact inverse.  
Note that we assume that $\D$ 
has discrete 
spectrum with eigenvalues $r_k$ for $k\in\bf Z$ where the spectral 
projection of $\D$ 
corresponding to the eigenvalue $r_k$ is denoted by $\Phi_k$. 
We suppose that $r_k$ is increasing with $k$ and if $k>0$ then $r_k>0$, 
and conversely, so that the zero eigenvalue, if it exists, corresponds to the 
index $k=0$.
Moreover, the eigenspaces $X_k=\Phi_k(X)$ are $F$-bimodules 
which sum to $X$
by hypothesis. We note that $X_0=\Phi_0(X)=\mbox{ker}\D.$

We observe that if $f$ is any real-valued function
defined (at least) on $\{r_k: k\in\bf{Z}\}$, the spectrum of $\D$, then
$f(\D)$ is the self-adjoint operator with domain:
$$\{x=\sum_k x_k \in X : \sum_k f(r_k)x_k\; {\rm converges}\; 
{\rm in}\; X\},$$
and is defined on this domain by $f(\D)x=\sum_k f(r_k)x_k.$ The convergence 
condition on the domain is equivalent to $\sum_k |f(r_k)|^2\la x_k|x_k\ra_X$
converges in $F$.

We further note that if $g:\R_+\to X$ is continuous and compactly supported
then for each $k\in\bf{Z}$, the function $g_k:=\Phi_k\circ g :\R_+\to X_k$
is continuous with $supp(g_k)\subseteq supp(g)$ and $g=\sum_k g_k$
converges in $\E$. Furthermore, if $g$ is smooth then so is each $g_k$
and $\p_t(g_k)=(\p_t(g))_k$ and by the previous sentence 
$\p_t(g)=\p_t(\sum_k g_k)=\sum_k\p_t(g_k).$


As both $\p_t\otimes 1$ and $1\otimes \D$ leave the subspaces
$L^2(\R_+)\otimes X_k$ invariant, in order to construct parametrices 
$Q_+$ and $Q_-$ for $T_+$ and $T_-$ we can begin by
considering homogeneous solutions $f_k$ to the equation
$$ T_{+,k}f_k=(\p_t+r_k)f_k=g_k$$
where $g_k$ is a smooth compactly supported function with values
in $X_k$ for each $k> 0.$
Setting
$$f_k(t)=Q_{+,k}(g_k)(t)=\int_0^te^{-r_{k}(t-s)}g_k(s)ds=\int_0^\infty
H(t-s)e^{-r_{k}(t-s)}g_k(s)ds,$$
where $H=\mathcal X_{\R_+}$ (the characteristic function of
$\R_+$) is the Heaviside function, we get a solution satisfying 
the boundary
conditions, as the reader will readily confirm.

Observe that for these homogeneous solutions our parametrix is given by
a convolution operator
$$f_k(t)=Q_{+,k}(g_k)(t)=(G_k*g_k)(t):=L_{G_k}g_k(t).$$
Here $G_k(s)=H(s)e^{-r_k s}\in L^1(\R)$, and
$\| G_k\|_1={1}/{r_k}$. Since the operator norm of $L_{G_k}$
on $L^2(\R)$ is bounded by $\|G_k\|_1$, we have
$$ \Vert Q_{+,k}\Vert =\Vert(L_{G_k}\otimes \Phi_k)\Vert_{End\E}\leq \|G_k\|_1\leq
{1}/{r_k}.$$
For $k<0$ we set
$$f_k(t)=Q_{+,k}(g_k)(t)=-\int_t^\infty e^{-r_k(t-s)}g_k(s)ds=
-\int_{-\infty}^\infty
\mathcal X_{(-\infty,0)}(t-s)e^{-r_k(t-s)}g_k(s)ds.$$
The verification that $T_+f_k=g_k$ is again straightforward, and
the solution is an
$L^2$-function with values in $\Phi_k(X)$ since it is given by the
convolution of an $L^1$ function and an $L^2$-function.

Later when we have defined $Q_{+,0}$ we will sum all the $Q_{+,k}$
to obtain the parametrix $Q_+$. At the moment we note that for a 
smooth compactly supported $g$ we have:
$$\left[Q_+(1\otimes (P-\Phi_0))(g)\right](t):=
\left[\sum_{k>0}Q_{+,k}(1\otimes\Phi_k)(g)\right](t)=\left[\sum_{k>0}Q_{+,k}g_k
\right](t)=\sum_{k>0}\int_0^te^{-r_{k}(t-s)}g_k(s)ds.$$
If we formally interchange the sum and the integral we get the equation:
$$\left[Q_+(1\otimes (P-\Phi_0))(g)\right](t)\mbox{`='}
\int_0^t\sum_{k>0}e^{-r_{k}(t-s)}\Phi_k(g)(s)ds=
\int_0^t e^{-\D(t-s)}(P-\Phi_0)(g(s))ds.$$
It is not hard to see that this convolution on the right actually converges to
the expression on the left in the norm of our module $L^2(\R_+)\otimes X.$

Similarly for the equation
$T_{-,k}f_k=(-\p_t+r_k)f_k=g_k$
we have the solutions
$$Q_{-,k}(g_k)(t)=\int_t^\infty e^{-r_k(s-t)}g_k(s)ds=
\int_{-\infty}^\infty \mathcal X_{(-\infty,0)}(t-s)e^{r_k(t-s)}g_k(s)ds,\ \ \ k> 0,$$
$$ Q_{-,k}(g_k)(t)=-\int_0^t e^{r_k(t-s)}g_k(s)ds =-\int_{-\infty}^\infty
H(t-s)e^{r_k(t-s)}g_k(s)ds,\ \ \ k<0.$$ Again this solution is given
by a convolution, and in all cases $k\neq 0$ we get
$\Vert Q_{\pm,k}(1\otimes\Phi_k)\Vert\leq{1}/{|r_k|}$.
We can get a similar operator convolution equation for 
$\sum_{k<0}Q_{+,k}g_k.$

Before proceeding we require a general lemma.

\begin{lemma}
\label{lm:abstract-aps}
Let $Y$ be a $C^*$-$F$-module and $Y_0\subseteq Y$ a dense
$F$-submodule. Let $T:Y_0\to Y_0$ be closable as a module mapping on
$Y$, with closure $\overline{T}$. Suppose there exists a bounded
module mapping $S$ on $Y$ such that
(1) $S(Y_0)\subset Y_0$, and 
(2) $ST=Id_{Y_0}$ and $TS|_{Y_0}=Id_{Y_0}$.
Then $S$ is one-to-one and $\overline{T}=S^{-1}:{\rm Image}\,(S)\to Y$,
${\rm dom}\,\overline{T}={\rm Image}\,(S)$,
$S\circ\overline{T}=Id_{{\rm dom}\,\overline{T}}$, and
$\overline{T}\circ S=Id_{Y}$.
\end{lemma}

\begin{proof} This is essentially just a careful check of the definitions
  of the domains and closures in question.
Let $y\in\mbox{dom}\,(\overline{T})$ so there exists a sequence
$\{y_n\}\subset Y_0$ converging to $y$ and $Ty_n\to \overline{T}y$
also. Now, since $S$ is bounded,
$$ y_n=STy_n\to S(\overline{T}y)\ \ \mbox{and}\ \ y_n\to y,$$
so $S(\overline{T}y)=y$ and
$S\circ\overline{T}=Id_{\mbox{dom}\,(\overline{T})}$. This also shows
$\mbox{dom}(\overline{T})\subset\mbox{Image}\,(S)$.

On the other hand, let $y=Sy'\in\mbox{Image}\,(S)$. Then $y'=\lim
z_n$, where $\{z_n\}\subset Y_0$, and so $y=Sy'=\lim Sz_n$. Since
$S:Y_0\to Y_0$, we see that  $\{Sz_n\}\subset
Y_0\subset\mbox{dom}\,(T)$, and so $z_n=TSz_n$ converges to $y'\in
Y$. Hence $y\in\mbox{dom}\,\overline{T}$ and $\overline{T}y=y'$. That
is $\mbox{Image}\,(S)\subset\mbox{dom}\,(\overline{T})$, and so they
are equal. Finally, $\overline{T}Sy'=\overline{T}y=y'$, and as $y'\in
Y$ was arbitrary, $\overline{T}S=Id_Y$. Hence $S$ is one-to-one, and
$\overline{T}=S^{-1}$.
\end{proof}

Returning to the operators $T_{\pm}$ and $Q_{\pm}$ on the module
$\E\ominus (1\otimes \Phi_0)\E$, we have the following preliminary
result. The proof is just a check of the hypotheses of the previous lemma.

\begin{cor}
\label{cr:doms-invses}
For $k\neq 0$, let $\E_k=L^2(\R_+)\otimes X_k$ and $\E_{k,0}\subset\E_k$ 
be the algebraic tensor product of
$$C_{00}^{\infty}(\R_+):=\{g\in C^\infty(\R_+):\;g(0)=0\;\;and\;\; supp(g)\;\;
is\;\;compact\}$$
with $X_k$. That is, $\E_{k,0}=C_{00}^{\infty}(\R_+)\odot X_k.$
Then $T_{\pm,k},\,Q_{\pm,k}$ map $\E_{k,0}$ to itself,
and are mutual inverses there. Hence
${\rm dom}\,(\overline{T_{\pm,k}})={\rm Image}\,(Q_{\pm,k})$, $Q_{\pm,k}$ is
one-to-one,  and the operators $\overline{T_{\pm,k}}$ and $Q_{\pm,k}$
  are mutually inverse (on appropriate subspaces).
\end{cor}
We extend this result by another application of Lemma \ref{lm:abstract-aps}:
\begin{cor}
\label{cr:bigger} Let the algebraic direct sum of the $\E_{k,0}$
with $k\neq 0$ be denoted
$$\E_{alg,0}:=\sum_{alg,k\neq0}\E_{k,0}=
\sum_{alg,k\neq0}C_{00}^{\infty}(\R_+)\odot X_k=
C_{00}^{\infty}(\R_+)\odot\sum_{alg,k\neq0} X_k$$
 Define $Q_{\pm}$ on $\E_{alg,0}$ as the algebraic
direct sum of the $Q_{\pm,k}$, and similarly for $T_{\pm}$.  Then
$Q_{\pm}$ extends to an operator on the completion, $\E_0$  
where it is bounded and one-to-one. Moreover, 
$\overline{T_{\pm}}=Q_{\pm}^{-1}:{\rm Image}(Q_{\pm})\to\E_0$ 
so that
$Q_{\pm}\circ\overline{T_{\pm}}=
Id_{{\rm {dom}}\overline{T_{\pm}}}$ and 
$\overline{T_{\pm}}\circ Q_{\pm}=
Id_{\E_0}.$
 We observe that 
$\E=\E_0\oplus (L^2(\R_+)\otimes X_0)$ as an internal orthogonal direct sum.
That is, $\E_0^\perp=(L^2(\R_+)\otimes X_0).$ 
\end{cor}

\subsection{The adjoint on $L^2(\R_+)\otimes X_0$ and self-adjointness of 
$\hat{\D}$}

On $L^2(\R_+)\otimes X_0$ the operator $T_{+,0}$ becomes
$\p_t\otimes Id_{X_0}$ with boundary conditions $\xi(0)=0$ while
$T_{-,0}=-\p_t\otimes Id_{X_0}$ with free boundary conditions, and it is
well-known that these two operators are mutual adjoints, cf
\cite[page 116]{L}. The parametrix $Q_{+,0}$ for $T_{+,0}$ is given by
$$Q_{+,0}(g)(t)=\int_0^t g(t)dt\;\;\mbox{for}\;\;g\in 
range(\overline{T_{+,0}}),$$
while the parametrix $Q_{-,0}$ for $T_{-,0}$ is given by
$$Q_{-,0}(g)(t)=-\int_t^{\infty} g(t)dt\;\;\mbox{for}\;\;g\in 
range(\overline{T_{-,0}}).$$
Of course, both $Q_{+,0}$ and $Q_{-,0}$ are unbounded operators and on
$L^2(\R_+)\otimes X_0$ we have:
$$\overline{T_{\pm,0}}Q_{\pm,0}=Id_{range(\overline{T_{\pm,0}})}\;\;\mbox{and}
\;\;Q_{\pm,0}\overline{T_{\pm,0}}=Id_{dom(\overline{T_{\pm,0}})}.$$
Letting $Q_{\pm}$ denote the (closure of the) direct sum of all the $Q_{\pm,k}$
we get the parametrix for $\overline{T_{\pm}}.$
\begin{prop}
\label{pr:self-adj} The adjoint of $\overline{T_\pm}:{\rm
dom}\,(\overline{T_\pm})\to\E$ is $\overline{T_\mp}$. Moreover,
$$\overline{T_{\pm}}Q_{\pm}=Id_{range(\overline{T_{\pm}})}\;\;\mbox{and}
\;\;Q_{\pm}\overline{T_{\pm}}=Id_{dom(\overline{T_{\pm}})}.$$

\end{prop}

\begin{proof}
In the following we write $T_\pm$ for the closure of $T_\pm$.
We write $T_\pm=T_\pm(1\otimes\Phi_0)\oplus T_\pm(1_{\E}-(1\otimes\Phi_0))$
and observe from our last comments that $(T_\pm(1\otimes
\Phi_0))^*=T_\mp(1\otimes\Phi_0).$

Restricting to $(1_{\E}-(1\otimes\Phi_0))\E=\E_0$ we have $Q_\pm^*=Q_\mp$. 
To see this,
recall that $Q_\pm$ is bounded, and so it suffices to check on the
dense submodule $\E_{alg,0}$ of Corollary \ref{cr:bigger}. For
$\xi,\,\eta\in\E_{alg,0}$, there is $\xi_0,\,\eta_0\in\E_{alg,0}$ such that
$\xi=T_\pm\xi_0$ and $\eta=T_\mp\eta_0$ ($\xi_0=Q_\pm\xi$ and
similarly for $\eta_0$). Then
\begin{align*}
(Q_\pm\xi|\eta)_\E&=(Q_\pm(T_\pm\xi_0)|T_\mp\eta_0)_\E=(\xi_0|T_\mp\eta_0)_\E\nno &=(T_\pm\xi_0|\eta_0)_\E\quad\mbox{by
symmetry}\nno &=(\xi|Q_\mp\eta)_\E.
\end{align*}
Hence $Q_\pm^*=Q_\mp$ on $(1_{\E}-(1\otimes\Phi_0))\E=\E_0$. In order to 
deduce from this a similar relation for the $T_\pm$ on $\E_0$ we need the 
following general considerations.

For a densely defined module map $T:\E_0\to\E_0$ we have the relation
between graphs
$$G(T^*)=[\nu(G(T))]^\perp=\nu[G(T)^\perp],$$
where $\nu:\E_0\oplus\E_0\to\E_0\oplus\E_0$ is the unitary given by
$\nu(x,y)=(y,-x)$, \cite[page 95]{L}.
Also for one-to-one module maps $Q$, $G(Q^{-1})=\theta(G(Q))$ where
$\theta(x,y)=(y,x)$ and $\theta \nu=-\nu\theta$.
So restricting $T_\pm$ to $\E_0$ we calculate:
\begin{align*} G(T_+^*)&=[\nu(G(T_+))]^\perp=[\nu(G(Q_+^{-1}))]^\perp\nno
&=[\nu(\theta(G(Q_+)))]^\perp=-[\theta(\nu(G(Q_+)))]^\perp=
-\theta[\nu(G(Q_+))^\perp]\nno
&=-\theta[G(Q_+^*)]=-\theta[G(Q_-)]=-[G(Q_-^{-1})]=-[G(T_-)]\nno
&=G(T_-).
\end{align*}
The same proof works for $T_-$, and so $T_\pm^*=T_\mp$ on all of
$\E$.
\end{proof}

The next step is to introduced the notion of {\bf extended solutions}.
In \cite{APS1}, the analogue of our module was introduced as a model
of a (product) neighbourhood of the boundary for a manifold-with-boundary.
Since the interest there, as here, was in the index of the operator
on the whole manifold-with-boundary, it was necessary to modify the
space of solutions considered to account for those functions on the boundary which
extended to interior solutions in a non-trivial way. Such functions are
not $L^2$ on this product description of the boundary, but are
bounded. Nevertheless they contribute to the index, and so we make a
definition.

\begin{defn} Let $(X,\D)$ be an unbounded odd Kasparov $A-F$-module.
Let $\E=L^2(\R_+)\otimes X$ be the $M(F,A)-F$-module defined above.
As seen in Lemma \ref{got-functions}, any element in the domain
of the operator $\p_t\otimes 1$ (free boundary conditions) is given by a 
uniformly continuous $X$-valued function $g$ which vanishes at $\infty$ and
the integral $\la g|g\ra_{\E}=\int_0^\infty \la g(t)|g(t)\ra_X dt$
converges in $F^+.$ We enlarge $\E$ to a space $\hat{\E}$
consisting of formal sums, $f=g+x$ where $g\in\E$ and $x\in X_0.$
For $g\in \rm{dom}(\p_t\otimes 1),$  the element $f=g+x$ is naturally
a function on $\R_+$ where $f(t)=g(t)+x$ and 
$\lim_{t\to\infty} f(t)=x\in X_0.$
We call such an $f$ an {\bf extended $L^2$-function}
and we may regard $f$ as a function $f:\R_+\to X$ 
with a limit:  $\lim_{t\to\infty} f(t):=f(\infty)$  such that
$f-f(\infty)$ is in $L^2(\R_+)\otimes X$ and $f(\infty)\in X_0$, that is,
$\D f(\infty)=0$. Note we reserve the terms 
{\bf extended $L^2$-function} and {\bf extended solution} to the case where
$f(\infty)\neq 0.$
\end{defn}

So, we have a new module $\hat{\E}=\{f=g+x\;|\;g\in\E\;\;\mbox{and}\;\;
x\in X_0\}.$ We let $F$ act on the left and right of this
extra copy of $X_0$ by its natural action. The $F$-valued 
inner product on $\hat{\E}$
is given by:
$$\la f+x|h+y\ra= \la f(t)|h(t)\ra_\E+\la x|y\ra_X.$$
 The left action of $M(F,A)$ on the
extra component $X_0$ is naturally defined to be {\bf zero}
since $M(F.A)$ consists of functions which vanish at $\infty.$ 
However, when we 
extend the left action to the unitization of $M(F,A)$ the added identity
will of course act as the identity on the extra copy of $X_0.$ While
$\D$ naturally acts as {\bf zero} on this extra copy of $X_0$, functions
$f(\D)$ act as multiplication by $f(0)$ so that
in particular, $P$ acts as the identity operator on this copy of $X_0$
and the operator, $\p_t$ naturally extends here as the {\bf zero}
operator. 

 We now modify our earlier definition of $\hat{X}$ to include
$\hat{\E}$ {\bf only} in the second component. Hence, by definition:
$$ \hat{X}=\bma \E\\ \hat{\E}\ema.$$
For the first
component any solution (i.e. element of the kernel of $T_+$)
necessarily vanishes on the boundary, and classically cannot
contribute to the index and the same situation persists in this
noncommutative setting.


We extend the action of $T_-$ to a map: $\hat{\E}\to\E$ via
$T_-(f+x)=T_-(f).$ Similarly we extend the action of $T_+$ to a map:
$\E\to\hat{\E}$ via $T_+(f)=T_+(f)+0$ and we extend the definitions of the actions of $Q_{+}$ and $Q_-$.
In order to emphasize the extension of $T_-$ we use the somewhat 
clumsy notation:
$$ \hat\D=\bma 0 & T_-\oplus 0\\ T_+ & 0\ema.$$
The addition of the zero map does not affect the adjointness
properties proved above, and so
$$ (T_-\oplus 0)^*=T_+\quad\mbox{and}\quad T_+^*=T_-\oplus 0.$$
Thus $\hat\D$ is self-adjoint.
We summarise this lengthy discussion.

\begin{prop}\label{prop:selfadj} Let $X$ be a right $C^*$-$F$-module, and
$\D:{\rm dom}\D\subset X\to X$ be a self-adjoint regular operator with
discrete spectrum. Then the
operator
$$\hat{\D}=\bma 0 & (-\p_t\otimes 1+1\otimes\D)\oplus 0\\
\p_t\otimes 1+1\otimes\D & 0\ema\;\;{\rm defined\;\; on}\;\; 
\bca \E\\ \hat{\E}\eca$$
satisfying APS boundary conditions as above is self-adjoint
and regular on $ \hat{X}=(\E\oplus \hat{\E})^T.$
\end{prop}

\begin{proof}
It remains only to show that $\hat\D$ is regular, namely $(1+\hat\D^2)$ has
dense range. We begin with
$\hat{\D}$ restricted to $(\E\oplus\E)^T$. We restrict ourselves further
to the invariant subspace $(\E_0\oplus\E_0)^T.$ To this end let $R=Q_+Q_-$. 
This is a bounded, positive endomorphism on $\E_0$
which is injective and has dense range (both $Q_+,\,Q_-$ are injective
with dense range, and are mutual adjoints by Proposition
\ref{pr:self-adj}). Hence the (unbounded) densely defined operator
$R^{-1}=(Q_+Q_-)^{-1}=Q_-^{-1}Q_+^{-1}=T_-T_+$ on $\E_0$
is a one-to-one positive operator which is onto. 
As the operator $R+1$ is bounded, positive
and (boundedly) invertible, it is surjective. Thus on
$\mbox{dom}\,(T_-T_+)$ consider the operator
$$(R+1)R^{-1}=1+R^{-1}=1+T_-T_+.$$
This is the composition of two surjective operators and so is
surjective
(on $\E_0$). Similar
comments apply to $1+T_+T_-$ (on $\E_0$). Thus $(1+\hat\D^2)$ restricted to
(its domain in) $(\E_0\oplus\E_0)^T$ maps onto $(\E_0\oplus\E_0)^T.$

Next, inside $\E$, we have $\E_0^\perp=L^2(\R_+)\otimes X_0$ and $\hat\D$ on
$(\E_0^\perp\oplus\E_0^\perp)^T$ is just 
$\bma 0 & -\p_t\\ \p_t & 0\ema \otimes 1_{X_0}.$
As regularity is automatic on $(L^2(\R_+)\oplus L^2(\R_+))^T$, 
we have regularity on all of $(\E\oplus\E)^T$.
Now, on $X_0\hookrightarrow\hat{\E}$, $\hat\D$ is defined as zero, so
$(1+\hat\D^2)|_{X_0}=1_{X_0}$, which is surjective. Putting the
pieces together, $1+\hat\D^2$ is surjective on $\hat{X}$.
\end{proof}

For use in the next proposition, we consider a more explicit discussion of 
regularity.
So we consider the equation
$$ \bma 1+T_-T_+ & 0\\ 0 & 1+T_+T_-\ema\bca
f_1\\f_2\eca=\bma 1-\p_t^2+\D^2 & 0\\ 0 & 1-\p_t^2+\D^2\ema\bca
f_1\\f_2\eca=\bca g_1\\g_2\eca.$$ Here we initially suppose each of
$(g_1,g_2)^T$ is in $C_{00}^{\infty}(\R_+)\odot\sum_{alg} X_k$.
With the exception of the
extra kernel term, such pairs are dense in $\hat X.$
We need to find $f=(f_1,f_2)^T$ in the domain of
$\hat\D^2$ satisfying this equation. In solving this equation we may
therefore assume that all terms are homogeneous, meaning that the general
solution is built from functions that map $\R_+$ to a single
eigenspace for $\D$, corresponding to the eigenvalue $r_k$. Thus the
equation we must solve, for given $(g_1,g_2)^T\in\hat{X}$, is
$$  \bma 1-\p_t^2+r_k^2 & 0\\ 0 & 1-\p_t^2+r_k^2\ema\bca
f_1\\f_2\eca=\bca g_1\\g_2\eca.$$ 
The boundary conditions are
$$ r_k\geq0\ \ \ \left\{\begin{array}{l}f_1(0)=0\\
  ((-\p_t+r_k)f_2)(0)=0\end{array}\right.,\quad
\qquad r_k<0\ \ \ \left\{\begin{array}{l}f_2(0)=0\\
((\p_t+r_k)f_1)(0)=0\end{array}\right.$$ 
We use the notation $\widehat{r_k}:=(1+r_k^2)^{1/2}$
as this term appears so often. The solution for $f_1$ is
$$ f_1(t)=(2\widehat{r_k})^{-1}\left(\int_t^\infty
e^{\widehat{r_k}(t-w)}g_1(w)dw +\int_0^t
e^{-\widehat{r_k}(t-w)}g_1(w)dw\right) +Ae^{-\widehat{r_k}t},$$ 
where for
$$r_k\geq 0,\ \ \ A=\frac{-1}{2\widehat{r_k}}\int_0^\infty
e^{-w\widehat{r_k}}g_1(w)dw, \ \ \mbox{and for }r_k<0,\ \ \ A=
\frac{1}{2\widehat{r_k}}\frac{\widehat{r_k}+r_k}{\widehat{r_k}-r_k}\int_0^\infty
e^{-w\widehat{r_k}}g_1(w)dw. $$ 
Observe that in terms of the Heaviside function $H$:
 \begin{align*}f_1(t)&=
\frac{1}{2\widehat{r_k}}\left(\int_{-\infty}^\infty
H^\perp(t-w)e^{\widehat{r_k}(t-w)}g_1(w)dw\right.\nno &+\left.
\int_{-\infty}^\infty H(t-w)e^{\widehat{r_k}(t-w)}g_1(w)dw
+\left\{\begin{array}{ll} -\langle
e^{-\widehat{r_k}\cdot},g_1(\cdot)\rangle
e^{-\widehat{r_k}t} & r_k\geq 0\\
+\frac{\widehat{r_k}+r_k}{\widehat{r_k}-r_k}\langle
e^{-\widehat{r_k}\cdot},g_1(\cdot)\rangle e^{-\widehat{r_k}t} &
r_k< 0\end{array}\right.\right). \end{align*} The point of this
observation is that it displays the integral as a convolution by an
$L^1$-function, plus a rank one operator, namely a multiple of the projection 
onto span$\{e^{-\widehat{r_k}t}\}$. Thus $f_1$ is an $L^2$-function.

For $f_2$ the situation is analogous. We have
$$ f_2(t)=(2\widehat{r_k})^{-1}\left(\int_t^\infty
e^{\widehat{r_k}(t-w)}g_2(w)dw +\int_0^t
e^{-\widehat{r_k}(t-w)}g_2(w)dw\right) +Be^{-\widehat{r_k}t},$$ 
where for
$$r_k< 0,\ \ \ B=\frac{1}{2\widehat{r_k}}\int_0^\infty
e^{-w\widehat{r_k}}g_2(w)dw, \mbox{ and for }  r_k\geq 0,\ \ \ B=
\frac{1}{2\widehat{r_k}}\frac{\widehat{r_k}-r_k}{\widehat{r_k}+r_k}\int_0^\infty
e^{-w\widehat{r_k}}g_2(w)dw. $$ 
Now we consider elements of $\hat X$
which only have a nonzero component in $X_0$. For such elements
$(0,0+x)^T$ we have
$$(1-\p_t^2+\D^2)x=(1-0+0)x=x,$$
so we have surjectivity for such elements. Now write a general
$g=(g_1,g_2+x)^T\in\hat{X}$ as
$$g=\bca g_1\\g_2+0\eca+\bca 0\\0+ x\eca.$$
Then the above solutions show that for any $g$ in a dense subspace of $\hat{X}$,
we can find
$f\in\mbox{dom}\,\hat\D^2$ with $(1+\hat\D^2)f=g$. Hence, we
have a second proof that $\hat\D$ is regular which we now exploit.

In the next result  APS boundary conditions 
mean that $\hat\D$ is defined on those
$\xi=(\xi_1\oplus\xi_2)^T$ in $(\E\oplus\hat{\E})^T$ such that
$ \hat\D\xi\in\hat{X},\ \ P\xi_1(0)=0,\ \ (1-P)\xi_2(0)=0.$
This is all well defined thanks to Lemma \ref{got-functions}.

\begin{prop}\label{thmone} Let $(X,\D)$ be an ungraded unbounded
Kasparov module for $C^*$-algebras $A,F$ with $F\subset A$ a
subalgebra satisfying $\overline{A\cdot F}=A$. Suppose that $\D$
also commutes with the left action of $F\subset A$, and that $\D$
has discrete spectrum. Then there is an unbounded graded Kasparov module
$$(\hat{X},\hat\D)=\left(\bca
\E \\ \hat{\E}\eca,
\bma 0 & T_-\\
T_+ & 0\ema\right)=\left(\bca
L^2(\R_+)\otimes X \\ \widehat{L^2(\R_+)\otimes X}\eca,
\bma 0 & -\p_t+\D\\
\p_t+\D & 0\ema\right)$$
(with APS boundary conditions) for
the mapping cone algebra $M(F,A)$. 
\end{prop}

\begin{proof} The most important observation is that the left action
of $M(F,A)$ on $\hat{X}$ preserves the APS boundary condition, and therefore
the domain of $\hat\D$  because for every $f\in M(F,A)$,
$f(0)\in F$ and hence commutes with the spectral projections defining the
boundary conditions. We note that to see that the action  of $M(F,A)$ 
on $\hat X$ is by {\bf bounded} module maps requires the strong boundedness
property of all adjointable mappings \cite{L} Proposition 1.2. 
We let $\A\subset A$ be the 
$*$-subalgebra of
$A$ such that for all $a\in\A$, $[\D,a]$ is bounded (on $X$) and
$a(1+\D^2)^{-1/2}$ is a compact endomorphism of $X$. We define the algebra
$$ \cM(F,\A)=\{f:\R_+\to\A:\
f(0)\in F\ \mbox{and}\ f\in C_0^\infty(\R_+)\;\;\mbox{and}\;\;[\hat\D,f]
\;\;\mbox{is\;\;bounded}\}.$$
We observe that the *-algebra of finite sums: 
$$\{\sum_if_i\otimes a_i: f_i\in C^\infty(\R_+)\;\;\mbox{and}\;\; 
f_i(0)=0\;\;\mbox{if}\;\;a_i\not\in F\}$$ 
is dense in $M(F,A)$ and is a *-subalgebra of $\cM(F,\A)$.

By Proposition \ref{prop:selfadj}, the operator $\hat\D$ is regular and
self-adjoint, so we
may employ the continuous functional calculus \cite{L},
to prove that $f(1+\hat{\D}^2)^{-1/2}$ is a compact endomorphism. It suffices to
show that
$f(1+\hat{\D}^2)^{-1}$ is compact. To see this, observe that 
$f(1+\hat{\D}^2)^{-1/2}$ is compact if and only if
$$ f(1+\hat{\D}^2)^{-1}f^*=f(1+\hat{\D}^2)^{-1/2}(1+\hat{\D}^2)^{-1/2}f^*$$
is compact
and this follows if $f(1+\hat{\D}^2)^{-1}$ is compact.
The latter follows by observing that from
our second proof of Proposition \ref{prop:selfadj} we have that each 
diagonal entry of
$$f(1+\hat{\D}^2)^{-1}\bma (1\otimes \Phi_k) & 0\\
0 & (1\otimes \Phi_k) \ema:=f(1+\hat{\D}^2)^{-1}((1\otimes \Phi_k)\otimes1_2)$$
can be expressed as a finite sum of
terms of the form $f(L_{g_k}\otimes\Phi_k)+f(R_k\otimes\Phi_k)$
where $L_{g_k}$ is convolution by an $L^1$-function and $R_k$ is a
rank one operator. 
We consider a single elementary tensor in the above subalgebra of $\cM(F,\A)$:
$f=h\otimes a$, where $a=a_1\cdot
b$, where $b\in F$ and $a_1\in A$. For such an elementary tensor the
diagonal entry is $(h\cdot L_{g_k}+h\cdot R_k)\otimes a_1\cdot b\Phi_k.$
Since $g_k$ is in $\LL^1$, the product $h\cdot L_{g_k}$ is a compact operator 
on $L^2(\R_+)$, and of course $hR_k$
is compact. Since $b(1+\D^2)^{-1/2}$ is a compact endomorphism on
$X$, it is straightforward to check that  $b\Phi_k$ is a compact
endomorphism. So as $End^0_{F}(L^2(\R_+)\otimes X)=End^0_\C(L^2(\R_+))\otimes
End^0_F(X)$, \cite{RW}[Corollary 3.38], the endomorphism
$$ B_k:= f((L_{g_k}+R_k)\otimes \Phi_k)=
(1\otimes a_1)( h(L_{g_k}+R_k)\otimes b\Phi_k) =(1\otimes a_1)C_k$$
is compact: indeed each $C_k$ is compact on $L^2(\R_+)\otimes X_k$. 
The importance of this
description is that $f(1+\hat{\D}^2)^{-1}=(1\otimes a_1)(\oplus_k
C_k)$ is a {\em direct sum} of compacts on $\oplus_k (L^2(\R_+)\otimes X_k)$ 
times the bounded operator $(1\otimes a_1)$.

The operator norm of $L_{g_k}$ on $L^2(\R_+)$ 
is bounded by the $L^1$-norm of $g_k$, and so
$$ \Vert L_{g_k}\Vert_{op}\leq\Vert
g_k\Vert_1=(1+r_k^2)^{-1/2}.$$ The norm of the rank
one operator $R_k$ on $L^2(\R_+)$ is given by Cauchy-Schwarz
as
$$ \Vert R_k\Vert_{op}\leq (2(1+r_k^2))^{-1}.$$
(This inequality is unaffected by multiplication by
$(\widehat{r_k}+|r_k|)/(\widehat{r_k}-|r_k|)$, 
so can be applied to both $r_k<0$ and
$r_k\geq 0$). Hence
\begin{align*} \Vert C_k\Vert_{op} &\leq \Vert h\Vert_{op}
\Vert L_{g_k}\Vert_{op} \Vert b\Vert_{op} + \Vert h\Vert_{op} \Vert
R_k\Vert_{op} \Vert b\Vert_{op}\nno &\leq \Vert h\Vert_{op} \Vert
b\Vert_{op}((1+r_k^2)^{-1/2}+(2(1+r_k^2))^{-1}).
\end{align*}
Since $1+r_k^2\to\infty$ as $|k|\to\infty$, the sequence of compact
endomorphisms $\{(1\otimes a_1)\sum_{-N}^NC_k\}$
converges in norm to $f(1+\hat{\D}^2)^{-1}$, which is therefore
compact. Since an arbitrary $f\in \cM(F,\A)$ is the norm limit of finite sums 
$\sum f_j\otimes a_j$ we see that $f(1+\hat{\D}^2)^{-1}$ is compact 
for general $f$ in the mapping cone algebra.

We can now show that we do indeed obtain a
Kasparov module.
First $V=\hat\D(1+\hat\D^2)^{-1/2}$ is self-adjoint. Also
$f(1-V^2)=f(1+\hat\D^2)^{-1}$
is a compact endomorphism for $f\in\cM(F,\A)$. Since
$V$ clearly anticommutes with the grading operator $\Gamma=\bma 1 &
0 \\ 0 & -1\ema$, we need only show that $[V,f]$ is compact for all
$f\in M(F,A)$. For $f$ a sum of elementary tensors (using smooth
functions), we may write this commutator as
$$[V,f]=[\hat\D,f](1+\hat\D^2)^{-1/2}+\hat\D[(1+\hat\D^2)^{-1/2},f]$$
Now for an elementary tensor $f\otimes a$, we get $[\hat\D,f\otimes a]
=\p f\otimes a +f\otimes[\D,a]$ and so the first term in the above equation
is compact. In the proof  of Proposition 2.4 of \cite{CP1} we have the formula:
\begin{eqnarray*} &&\hat\D[(1+\hat\D^2)^{-1/2},f]\\
&=&\frac{1}{\pi}\int_0^\infty \lambda^{-1/2}
\{\hat\D(1+\hat\D^2+\lambda)^{-1/2}\left\{(1+\hat\D^2+\lambda)^{-1/2}[f,\hat\D]
(1+\hat\D^2+\lambda)^{-1/2}\right\}\hat\D(1+\hat\D^2+\lambda)^{-1/2}\\
&+& \hat\D^2(1+\hat\D^2+\lambda)^{-1}[f,\hat\D](1+\hat\D^2+\lambda)^{-1}\}
d\lambda.\end{eqnarray*}
where the integral converges in operator norm and we have grouped the terms in 
the integrand so that they are clearly compact by the discussion above. It
follows that $[V,f]$ is a compact endomorphism for $f$ a sum of elementary 
tensors. Since these are norm dense in $M(F,A)$ and $V$ is bounded, 
$[V,f]$ is compact for all $f\in M(F,A)$. So we have an even
Kasparov module for $(M(F,A),F)$ with an unbounded representative for
$(\cM(F,\A),F)$.
\end{proof}

{\bf Remark}. It should be noted that in this
context, discreteness of the spectrum of $\D$ does NOT imply that
$(1+\D^2)^{-1/2}$ is a compact endomorphism. We are assuming that we
have a Kasparov module, so that for all $a\in A$ $a(1+\D^2)^{-1/2}$ is
a compact endomorphism, but these two compactness conditions are not
equivalent unless $A$ is unital. Kasparov modules corresponding to
infinite graphs provide examples of this phenomenon, \cite{pr}.

\section{Equality of the index pairings from the Kasparov modules.}
We formulate our main theorem in this Section demonstrating how
even and odd Kasparov modules give equal
index pairings.

We recall that given a partial isometry $v\in A$
with range and source projections in
$F$ (observe this includes unitaries in $A$), we defined
$ v_1=\bma 1-vv^* & v\\ v^* & 1-v^*v\ema.$
This is a self-adjoint unitary in $M_2(\tilde A)$, 
and hence there exists a norm continuous path of self-adjoint unitaries
in $M_2(\tilde A)$ from $v_1$ to the identity.  We choose the path
$$v_1(t)=\frac{1}{2}(e^{2i\tan^{-1}(t)}(v_1-1_2)+(v_1+1_2)),$$
so that $v_1(0)=v_1$ and $v_1(\infty)=1_2$.
Now define a projection $e_v(t)$ over $\tilde{M}(F,A)$ by
$$e_v(t)=v_1(t)\bma 1 & 0\\ 0 & 0\ema v_1(t)^*=\bma
1-\frac{1}{1+t^2}vv^* & \frac{-it}{1+t^2}v\\ \frac{it}{1+t^2}v^*&
\frac{1}{1+t^2}v^*v\ema,$$
where we have used some elementary trigonometry
to simplify the expressions. It is important to observe that this is a finite 
sum of elementary tensors $\sum f_j\otimes a_j$ with
$f_j$ smooth and square integrable or $f_j-f_j(\infty)$ smooth and square
integrable. As such it maps $(\hat{\E}\oplus\hat{\E})^T$ to itself
and leaves $(\E\oplus\E)^T$ invariant.

The difference of classes
$$ [e_v(t)]-\left[\bma 1 & 0\\0&0\ema\right]$$
lies in $K_0(M(F,A))$: see Lemma \ref{Im:Put} and the discussion preceding
it, as well as \cite{Put}.
Let $e=\bma 1 & 0 \\0 & 0 \ema$, a constant function, then the index pairing 
of $[v]\in K_0(M(F,A))$ with $[(\hat{X},\hat\D)]$ is
$$\la [e_v]-[e],[(\hat{X},\hat\D)]\ra
:={\rm Index}(e_v(\hat\D\otimes 1_2)e_v)-{\rm Index}(e(\hat\D\otimes 1_2)e)
\in K_0(F).$$

\begin{rems*} To explain this notation we review {\bf even} index theory. 
On $\bma\E\\\hat{\E}\ema$, 
$\hat\D=\bma  0 & T_-\\T_+ & 0\ema$ 
while the grading operator
$\Gamma=\bma 1 & 0\\0 & -1 \ema.$ That is $\hat\D$ is {\bf odd} while the 
action of $M(F,A)$ is {\bf even}, i.e., diagonal. Then, on 
$\bma\E\otimes\C^2\\\hat{\E}\otimes\C^2\ema$ we have:
$\hat\D\otimes 1_2=\bma \hat\D & 0 \\ 0 & \hat\D \ema$ and
$\Gamma\otimes 1_2=\bma \Gamma & 0 \\ 0 & \Gamma \ema$ while 
$e_v=\bma f & g \\ h & k \ema \in M_2(M(F,A))$ acts as 
$\bma f\otimes 1_2 & g\otimes 1_2 \\ h\otimes 1_2 & k\otimes 1_2 \ema.$ 
Let $([\E\oplus\hat{\E}]\oplus[\E\oplus\hat{\E}])^T
\cong ([\E\oplus\E]\oplus[\hat{\E}\oplus\hat{\E}])^T$ be the obvious 
unitary equivalence. Under this equivalence $\hat\D\otimes1_2$ becomes
$\bma 0 & T_-\otimes1_2 \\ T_+\otimes1_2 & 0 \ema$, while
$e_v=\bma f & g \\ h & k \ema \in M_2(M(F,A))$ acts as 
$\bma e_v & 0 \\ 0 & e_v\ema$.
Also, $\mbox{Index}(e_v(\hat\D\otimes1_2)e_v)$
really means the index of the lower corner operator of 
$\bma e_v & 0 \\ 0 & e_v\ema(\hat\D\otimes1_2)\bma e_v & 0 \\ 0 & e_v\ema=
\bma 0 & e_v(T_-\otimes1_2)e_v \\ e_v(T_+\otimes1_2)e_v & 0 \ema$:
$$e_v\bma T_+ & 0\\ 0 & T_+\ema e_v:\;\;\mbox{as\;\;a\;\;mapping}\;\;
e_v\bma \E\\ \E\ema \to e_v\bma\hat{\E}\\ \hat{\E}\ema.$$
That is we must compute both:
$${\rm ker}(e_v(T_+\otimes 1_2)e_v)\subseteq e_v\bma \E\\ \E\ema\;\;\;{\rm and}
\;\;\;{\rm ker}(e_v(T_-\otimes 1_2)e_v)\subseteq e_v\bma \hat{\E}\\ \hat{\E}\ema
\subseteq \bma \hat{\E}\\ \E\ema.$$
Similarly, $\mbox{Index}(e(\hat\D\otimes 1_2)e)$ 
means the index of the lower
corner operator:
$e\bma T_+ & 0\\ 0 & T_+\ema e$, that is, $T_+$
as a mapping from $\E\to\hat{\E},$
which we will write as $\mbox{Index}(\hat\D)$.
With this reminder,  and the convention that   if $T$ is an operator on the module $Y$, we write
$T_k$ for $T\otimes 1_k$ on the module $Y\otimes \C^k$,        
we now state our key result.
\end{rems*}

\begin{thm}\label{mainresult}
 Let $(X,\D)$ be an ungraded unbounded Kasparov module
for the (pre-) $C^*$-algebras $\A\subset A,F$ with $F\subset A$ a
subalgebra satisfying $\overline{A\cdot F}=A$. Suppose that $\D$
also commutes with the left action of $F\subset A$, and that $\D$
has discrete spectrum. Let $(\hat{X},\hat\D)$ be the unbounded
Kasparov $M(F,A),F$ module of Proposition \ref{thmone}.
Then for any unitary $u\in M_k(A)$ such that 
$P_k$ and $(\Phi_0)_k$ both commute with
$u\D_k u^*$ and $u^*\D_k u$
we have the following equality of
index pairings with values in $K_0(F)$:
\bean \la [u],[(X,\D)]\ra&:=&{\rm Index}(P_ku^*P_k)
=
{\rm Index}(e_u(\hat\D_k\otimes1_2) e_u)-{\rm Index}(\hat\D_k)\nno
&=:&
\la [e_u]-\left[\bma 1 & 0\\ 0 & 0\ema\right],[(\hat{X},\hat\D)]\ra
\in K_0(F).\eean
Moreover, if $v$ is a  partial isometry, $v\in M_k(\A)$, with
$vv^*,v^*v\in M_k(F)$ and such that 
$P_k$ and $ (\Phi_0)_k$ {both  commute  with} 
$v\D_k v^*$ and $ v^*\D_k v$
we have
\bea\la [e_v]-\left[\bma 1 & 0\\ 0 & 0\ema\right],[(\hat{X},\hat{\D})]\ra&=&
-{\rm Index}(PvP:v^*vP(X)\to vv^*P(X))\in K_0(F)\nno
&=&{\rm Index}(Pv^*P:vv^*P(X)\to v^*vP(X))\in K_0(F).\label{pisomindex}\eea
\end{thm}

\begin{rems*}(1) In the last statement 
we really are taking a Kasparov product when we
consider
$$K_0(M(F,A))\times KK^0(M(F,A),F)\to K_0(F).$$
 Hence the index
is well-defined, depends only on the class of $[e_v]-[1]=[v]$
and the class of the `APS Kasparov module'. \\
\noindent (2) We note that our hypothesis that $P$ and $\Phi_0$
 commute with 
 $v^*\D v$ is equivalent to $P$ and $\Phi_0$ 
 commuting with
$v^*dv$ since $P,\,\Phi_0$ 
commute with $\D$ and with $v^*v$. Thus $P,\,\Phi_0$ 
commute 
with all functions of $v^*\D v$, and in particular with each spectral
projection $v^*\Phi_k v$. Similarly, 
the first set of commutation relations
imply that $\D$ and all
of $\D$'s spectral projections commute with $v^*Pv$ and 
$v^*\Phi_0v$.\\
\noindent (3) Whether every class $[v]\in K_0(M(F,A))$ possesses 
a representative satisfying the hypotheses of the theorem is unknown to us 
in general. Just as with the issues of regularity, it {\em may} be that one 
can always homotopy $v$ and/or $(X,\D)$ so that the hypotheses are satisfied. 
We leave this issue for future work, noting that for the applications we have 
in mind the hypotheses are satisfied.\\
\noindent (4)
With regards to the regularity of $PvP$ (in the sense of having a pseudoinverse
\cite[Definition 4.3,]{GVF}), we observe that since $P$ commutes
with $v^*Pv,$ the operator $PvP$ is regular as an operator from $v^*vP(X)$
to $vv^*P(X)$,
where the pseudoinverse of $PvP$ is provided by $Pv^*P$. That is,
$(PvP)(Pv^*P)(PvP)=PvP$ and $(Pv^*P)(PvP)(Pv^*P)=Pv^*P.$ Thus our hypotheses 
guarantee the regularity of $PvP$, and the independence of the index of $PvP$ 
on which regular `amplification' we take gives some evidence that the 
hypotheses may be relaxed.
\end{rems*}
The proof of Theorem \ref{mainresult} will occupy the rest of the Section.

\subsection{Preliminaries}

As is usual for an index calculation such as this, we will assume without 
loss of generality (by replacing $\A$ by $M_k(\A)$ if necessary) that 
the partial isometry $v$ lies in $\A$.
To begin the proof it  is helpful to write $e_v$ as an orthogonal sum of
subprojections in $\mathcal{L}(\hat{X}\oplus\hat{X})$ which, of course, 
commute with $e_v$:
\begin{eqnarray}
\label{orthod}e_v=
\bma\frac{t^2}{1+t^2}vv^*& \frac{-it}{1+t^2}v\\ 
\frac{it}{1+t^2}v^*&
\frac{1}{1+t^2}v^*v\ema + \bma 1-vv^* & 0 \\ 0 & 0\ema :=\widehat{e_v}+e^0_v.
\end{eqnarray}

Note that to prove the Theorem
it suffices to demonstrate the equality in Equation
(\ref{pisomindex}), and that is what we shall do.
Using the decomposition of $e_v$ into orthogonal subprojections in 
(\ref{orthod}) an elementary calculation now gives:

\begin{lemma}\label{inmodule}
{\rm (1)} Let $\xi=\bma\xi_1\\\xi_2\ema\in
\bma\E\\\E\ema$. Then $\xi\in e_v\bma\E\\\E\ema$ if and only if
$v^*v\xi_2=\xi_2$ and $vv^*\xi_1=-itv\xi_2.$
In this case by Equation (\ref{orthod}) we get an orthogonal decomposition:
\begin{eqnarray}\label{decomp} \bca\xi_1\\\xi_2\eca=e_v \bca\xi_1\\\xi_2\eca=
\widehat{e_v}\bca\xi_1\\\xi_2\eca + e^0_v\bca\xi_1\\\xi_2\eca =
\bca \eta_1\\\xi_2\eca+\bca
\zeta_1\\0\eca,\end{eqnarray}
where $\eta_1=vv^*\xi_1=-itv\xi_2$ and $\zeta_1=(1-vv^*)\xi_1$; and {\bf both}
$\bma\eta_1\\\xi_2\ema$ and $\bma\zeta_1\\0\ema$ lie in $e_v\bma\E\\\E\ema$.\\
\noindent {\rm (2)} The same statement (mutatis mutandis) holds for 
$\xi=\bma\xi_1\\\xi_2\ema\in\bma\hat{\E}\\\hat{\E}\ema$
\end{lemma}

In order to solve the differential equations to find the index in the Theorem we need the commutation relations recorded in the following lemma.

\begin{lemma}\label{lm:subspaces-preserved} 
The operators $v^*\D v$, $v^*v\D$ and $v^*dv$ 
preserve the subspaces of $v^*v(X)$ (intersected with the 
appropriate domains where necessary) given by
$ v^*QvP(X),\ \ v^*Qv(1-P)(X),$
where $Q$ is any of the projections 
$P,\,P-\Phi_0,\,1-P,\,1-P+\Phi_0,\,\Phi_0$.
\end{lemma} 

\begin{proof} In the remarks after the statement of 
Theorem \ref{mainresult}, we noted that all spectral projections of $v^*v\D$ commute 
with the projections $v^*Qv$ with $Q$. As $v^*v\D$ 
also commutes with $P$ and $1-P$, $v^*v\D$ preserves
these subspaces. 
Likewise, $v^*\D v$ commutes with $v^*Q'v$ for {\em any} spectral 
projection $Q'$ of $\D$, and by the hypotheses on $v$, $v^*\D v$ 
commutes with  $P$ and so $1-P$. Thus $v^*\D v$ preserves all these 
subspaces.
The result for $v^*dv=v^*\D v-v^*v\D$ follows immediately.
\end{proof}

\subsection{Simplifying the equations}
The main consequence of Lemma \ref{inmodule} is  that we can consider two 
orthogonal subspaces of solutions separately and this greatly reduces the 
complexity of our task. In this subsection we will cover the $T_+$ case: 
${\rm ker}(e_v(T_+\otimes 1_2)e_v)$.

We observe that 
$(\p_t +\D)\otimes 1_2$ commutes with the projection
$\bma 1-vv^* & 0 \\ 0 & 0 \ema$ (which is $\leq e_v$). 
Thus with $Q_+$ the parametrix for $T_+=\p_t+\D$ constructed earlier we have
\bean &&\bma (1-vv^*) & 0\\ 0 & 0\ema(Q_+\otimes 1_2)\bma (1-vv^*) & 0\\
 0 & 0\ema e_v((\p_t+\D)\otimes 1_2)
e_v\bma 1-vv^* & 0\\ 0 & 0\ema \\
&&=\bma (1-vv^*) & 0\\ 0 & 0\ema(Q_+\otimes 1_2)((\p_t+\D)\otimes 1_2)
\bma 1-vv^* & 0\\ 0 & 0\ema  \\
&&=\bma (1-vv^*) & 0\\ 0 & 0\ema(Id\otimes 1_2)
\bma 1-vv^* & 0\\ 0 & 0\ema =\bma 1-vv^* & 0\\ 0 & 0\ema
\eean
Thus the kernel is $\{0\}$ on this subspace, and so we need only
calculate the kernel on the range of $\widehat{e_v}.$
Using the notation
$da:=[\D,a]$ and recalling that $vv^*$ and $v^*v$ commute with $\D$,
so that $v^*vdv^*=dv^*$ and $vv^*dv=dv$ we now obtain:
\begin{eqnarray*}&&\widehat{e_v}[(\p_t+\D)\otimes 1_2]\widehat{e_v}\\
&=&\bma \frac{t}{(1+t^2)^2}vv^* + \frac{t^2}{1+t^2}vv^*(\p_t +\D)
+\frac{t^2}{(1+t^2)^2}vdv^* & \frac{it^2}{(1+t^2)^2} v +\frac{-it}{1+t^2}v
(\p_t+\D) + \frac{-it^3}{(1+t^2)^2}dv \\ \frac{i}{(1+t^2)^2}v^* + 
\frac{it}{1+t^2}v^*(\p_t+\D) + \frac{it}{(1+t^2)^2}dv^* &
\frac{-t}{(1+t^2)^2}v^*v + \frac{1}{1+t^2}v^*v(\p_t+\D) +\frac{t^2}{(1+t^2)^2}
v^*dv \ema \\
&=&\frac{1}{1+t^2}\bma t^2vv^*(\p_t+\D) & -itv(\p_t+\D) \\ itv^*(\p_t+\D) &
v^*v(\p_t+\D)\ema + \frac{1}{(1+t^2)^2} \bma tvv^* + t^2 vdv^* &
it^2v -it^3dv \\ iv^* + itdv^* & -tv^*v + t^2v^*dv\ema.
\end{eqnarray*}
Using this formula, 
we obtain
$$\widehat{e_v}\bma (\p_t+\D) & 0 \\ 0 & (\p_t+\D)\ema\widehat{e_v}\bma \xi_1 \\
\xi_2\ema =\bma -itv(\p_t+\D)\xi_2 -\frac{it^2}{1+t^2}v\xi_2 -\frac{it^3}{1+t^2}
dv\xi_2 \\ (\p_t+\D)\xi_2 +\frac{t}{1+t^2}\xi_2 +\frac{t^2}{1+t^2}v^*dv\xi_2
\ema.$$
Since this vector is also in the range of $\widehat{e_v}$ we check that 
the first coordinate is $-itv$ times the second coordinate as required by 
Lemma \ref{inmodule}.
We may rewrite the second coordinate in the preceding equation:
$$\rho_2(t)=(\p_t+\D)\xi_2 +\frac{t}{1+t^2}\xi_2  +v^*dv\xi_2 
-\frac{v^*dv}{1+t^2}\xi_2$$
using $\xi_2=v^*v(\xi_2),$ and $1-1/(1+t^2)=t^2/(1+t^2)$ as:
$$\rho_2(t)=
\left(\frac{1}{\sqrt{1+t^2}}\p_t\circ\sqrt{1+t^2}
+v^*v\D  + \frac{t^2v^*dv}{1+t^2}\right)\xi_2 =:(\tilde{\D}_v + V)\xi_2$$
where $\tilde{\D}_v=\left( 
\frac{1}{\sqrt{1+t^2}}\p_t\circ\sqrt{1+t^2}+v^*v\D \right)$
and $V=\frac{t^2}{1+t^2}\otimes (v^*dv):=V_0\otimes (v^*dv).$
So in order to compute the kernel of 
$\widehat{e_v}[(\p_t+\D)\otimes 1_2]\widehat{e_v}$ acting on the range of
$\widehat{e_v}$, it suffices to compute the kernel of $\tilde{\D}_v + V$ acting
on vectors $\xi_2\in {\rm dom}(\tilde{\D})$ satisfying $v^*v(\xi_2)=\xi_2$
and $t\xi_2\in L^2(\R_+)\otimes X$. {\bf In the $T_+$ case only,}
such vectors are precisely those $\xi_2$ in ${\rm dom}(\tilde{\D})$ which lie in 
$L^2(\R_+,(1+t^2)dt)\otimes v^*v(X)$. We make the {important observation}
that $\tilde{\D}_v$ is naturally a densely defined closed operator on 
$L^2(\R_+,(1+t^2)dt)\otimes v^*v(X)$ completely analogous to the 
operator $T_+=\p_t +\D$ of Section 
\ref{secAPS} which acts on $L^2(\R_+)\otimes X.$

Now we consider boundary values.
For the equation $e_v((\p_t+\D)\otimes 1_2) e_v\xi=0$ we want to impose 
the boundary
condition
$e_v(0)(P\otimes 1_2)e_v(0)\xi(0)=0$ where $P$ is the non-negative spectral 
projection for $\D$. This
projection is
$$e_v(0)\bma P & 0\\0 & P\ema e_v(0)=\bma (1-vv^*)P & 0\\0 &v^*vP\ema.$$

 Observe that our boundary projection is also the non-negative
spectral projection of $e_v(0)(\D\otimes 1_2) e_v(0)$.
As noted above, the only solution which lies in the range of 
$e_v^0 (P\otimes 1_2)e_v^0=\bma (1-vv^*)P & 0\\0 & 0\ema$ is the zero solution, 
for which this condition is automatically satisfied. Hence, we need not 
concern ourselves any further with this subcase.

\subsection{Solutions, integral kernels and parametrices}

In the following we make some notational simplifications. {\bf We replace 
$v^*v\D$ by $\D$}, and similarly for other operators, since everything 
commutes with $v^*v$ and we will always be working on the subspace 
$v^*v(X)$.  In the notation of the previous subsection we aim to find
the solutions of $(\tilde{\D}_v+V)\rho=0$ on
$L^2(\R_+,(1+t^2)dt)\otimes v^*v(X)$. 

We will break our space up 
into orthogonal pieces preserved by $\tilde{\D}_v+V$. 
We first split our space as the image of $1\otimes P$ and $1\otimes(1-P)$.
On the image of $1\otimes P$ we define a two parameter family of 
bounded operators which will be the integral kernel of a local left inverse for 
$\tilde{\D}_v+V$ on this space. The reason for our notation 
$\tilde{\D}_v+V$ is that we regard $V$ as a (time dependent) 
perturbation, and we will define our integral kernels using a variant 
of the Dyson expansion for time dependent Hamiltonians, \cite[X.12]{RS}.

So for $t\geq s\geq 0$ define an operator on $Pv^*v(X)$ by
$$ U(t,s)=e^{-(t-s)P\D}+\sum_{n=1}^\infty (-1)^n
\int_s^t\int_s^{t_1}\cdots\int_s^{t_{n-1}} 
e^{-(t-t_1)P\D}V(t_1)e^{-(t_1-t_2)P\D}\cdots V(t_n)e^{-(t_n-s)P\D}{\bf dt},$$
where we write:  $dt_n\cdots dt_2dt_1 ={\bf dt},$ and where $P\D$ really means
$\D$ restricted to $Pv^*v(X).$

\begin{lemma}\label{lm:First-kernel} For all $t\geq s\geq 0$ the integrals and the 
infinite sum defining $U(t,s)$ converge absolutely in the operator norm 
on the space $Pv^*v(X).$
For all $t\geq s\geq 0$ we have 
$$\Vert U(t,s)\Vert\leq \Vert e^{-(t-s)P\D}\Vert\, e^{(t-s)\Vert v^*dv\Vert}.$$
Moreover $U(t,s)$ satisfies the differential equations
$$ \frac{d}{dt}U(t,s)=-(\D+V(t))U(t,s)\;\;\;\;{\rm and}\;\;\;\; 
\frac{d}{ds}U(t,s)=U(t,s)(\D+V(s)).$$
\end{lemma}

\begin{proof}
To see the convergence and the norm inequality, we use the crude 
estimate $\Vert V(t)\Vert\leq \Vert v^*dv\Vert$
together with the equalities: 
$$\Vert e^{-(t_k-t_{k+1})P\D}\Vert=\Vert e^{-P\D}\Vert^{(t_k-t_{k+1})}\;\;\;
\mbox{and}\;\;\;\int_s^t\int_s^{t_1}\cdots\int_s^{t_{n-1}}{\bf 1}\;dt_n\cdots dt_1
=(t-s)^n/n!,$$
to obtain the inequality:
$$\Vert U(t,s)\Vert\leq \Vert e^{-(t-s)P\D}\Vert\sum_{n-0}^{\infty}
\frac{(t-s)^n\Vert v^*dv\Vert^n}{n!}= 
\Vert e^{-(t-s)P\D}\Vert\, e^{(t-s)\Vert v^*dv\Vert}.$$
Differentiating formally 
yields the two differential equations but to see that the difference 
quotients converge in operator norm to the formal derivative takes a
little effort. For example, using the mean value theorem and the 
functional calculus for unbounded self-adjoint operators, one shows that
for any $f\in C^{(2)}(\R_+)$  which satisfies 
$x^2|f^{\prime\prime}(x)|\leq C$ for all $x\in\R_+$ we have:
$\frac{d}{dt}(f((at+b)\D))=a\D f^{\prime}((at+b)\D)$ when $(at+b)>0$
with norm convergence 
of the 
difference quotient. Applying this to $f(x)=e^{-x}$ for $(t-s)>0$ we get 
$\frac{d}{dt}e^{-(t-s)\D}=-\D e^{-(t-s)\D}$, and 
$\frac{d}{ds}e^{-(t-s)\D}=\D e^{-(t-s)\D}.$

As for differentiating the integral terms, formally one uses a product rule
which technically is invalid as one term is unbounded; however, by using the 
product rule trick of adding in a term and subtracting it out, one shows the
formal calculation works. Since the original series and the series for
the derivatives converge uniformly and absolutely, we are done. 
\end{proof}

Using these results we now construct a (local) left inverse for 
$(\tilde{\D}_v+V)(1\otimes P)$.
We define for any
$t\geq 0$
and continuous function
$\rho\in (L^2(\R_+,(1+t^2)dt)\otimes Pv^*v(X))$, $$(\tilde{Q}\rho)(t)
:=\frac{1}{\sqrt{1+t^2}}\int_0^tU(t,s)\sqrt{1+s^2}\rho(s)ds.$$
Observe that $(\tilde{Q}\rho)(0)=0$, and 
is differentiable. 
First we need an elementary operator-theoretic lemma.

\begin{lemma}\label{lm:operator}
Let $T$ be a closed densely defined operator on a Banach space $B$ and let
$\mathcal{S}\subseteq {\rm dom}(T)$ be a dense subspace of ${\rm dom}(T)$
in the domain norm. Let $A:{\rm dom}(T)\to B$ be a bounded operator
in the ${\rm dom}(T)$ norm, and let $Q$ be a densely defined closable linear 
operator whose
domain contains $T(\mathcal{S})$ and such that 
$QT=1_{\mathcal{S}} + A_{|_{\mathcal{S}}}.$ Then,
{\rm range}$(T)\subseteq {\rm dom}(\overline{Q})$
and
$\overline{Q}T=1_{{\rm dom}(T)} + A.$
\end{lemma}
\begin{proof}
Let $Tx\in{\rm range}(T),$ so there exists a sequence $\{x_n\}$ in
$\mathcal S$ with $x_n\to x$ and $Tx_n\to Tx.$ But then, the fact that
$\lim_n Tx_n=Tx$ and $\lim_{n}Q(Tx_n)=\lim_{n}(x_n+A(x_n))=x+A(x)$ implies that
$Tx\in{\rm dom}(\overline{Q})$ and $\overline{Q}(Tx)=x+A(x).$
\end{proof}

\begin{lemma}\label{lm:no-solns-1} 
The equation $(\tilde{\D}_v+V)\rho=0$ has no nonzero solutions in 
$$(L^2(\R_+,(1+t^2)dt)\otimes Pv^*v(X)).$$
\end{lemma}

\begin{proof}
Fix $M>0$ and let $E_{M}$ be the
orthogonal projection of $L^2(\R_+,(1+t^2)dt)$ onto the subspace
$L^2([0,M],(1+t^2)dt).$ Then $E_{M}\otimes 1$ is the orthogonal
projection of
 $(L^2(\R_+,(1+t^2)dt)\otimes Pv^*v(X))$ onto
the subspace $(L^2([0,M],(1+t^2)dt)\otimes Pv^*v(X)).$
Now, we see that $\tilde{Q}$ defines a linear operator
on the dense subspace of 
$(L^2([0,M],(1+t^2)dt)\otimes Pv^*v(X))$ consisting of
continuous functions, 
call it $\tilde{Q}_M.$
This operator has a densely defined adjoint defined on the same
subspace, $\tilde{Q}_M^\#,$ given by the formula:
$$(\tilde{Q}_M^{\#}\rho)(t)
:=\frac{1}{\sqrt{1+t^2}}\int_t^M U(s,t)^*\sqrt{1+s^2}\rho(s)ds.$$
Thus, $\tilde{Q}_M$ is not only densely defined, but also closable on
$(L^2([0,M],(1+t^2)dt)\otimes Pv^*v(X)).$

The smooth functions $\rho$ in the domain of 
$(\tilde{\D}_v+V)(1\otimes P)$ form a 
domain-dense subspace and
$$(E_M\otimes 1)(\tilde{\D}_v+V)(\rho)=(\tilde{\D}_v+V)(E_M\otimes 1)(\rho)
\in{\rm dom}(\tilde{Q}_M).$$
Let $\rho_M=(E_M\otimes 1)(\rho),$ fix $t\in [0,M]$ and calculate:
\begin{align*} 
(\tilde{Q}_M(\tilde{\D}_v+V)\rho_M)(t)&=(\tilde{Q}_M(\tilde{\D}_v+V)\rho)(t)= 
(\tilde{Q}(\tilde{\D}_v+V)\rho)(t)\nno
&=\frac{1}{\sqrt{1+t^2}}\int_0^tU(t,s)\left(\p_s(\sqrt{1+s^2}\rho(s))+
\sqrt{1+s^2}(\D+V(s))\rho(s)\right)ds\nno
&=\frac{1}{\sqrt{1+t^2}}\int_0^t\p_s(U(t,s)\sqrt{1+s^2}\rho(s))ds 
-\frac{1}{\sqrt{1+t^2}}\int_0^t(\p_s U(t,s)) \sqrt{1+s^2}\rho(s)ds\nno
&\qquad\qquad\qquad\qquad\qquad
+\frac{1}{\sqrt{1+t^2}}\int_0^tU(t,s)\sqrt{1+s^2}(\D+V(s))\rho(s)ds\nno
&=\rho(t)-\frac{1}{\sqrt{1+t^2}}U(t,0)\rho(0) =\rho(t)=\rho_M(t),
\end{align*}
As $\rho(0)=P(\rho(0))=0$ the previous lemma implies that
$(\tilde{\D}_v+V)(E_M\otimes P)$ is injective and 
$(\tilde{\D}_v+V)\rho=0$ has no nonzero local solutions on $[0,M]$ 
for any $M>0.$ Hence, $(\tilde{\D}_v+V)\rho=0$ has no nonzero global solutions
in $(L^2(\R_+,(1+t^2)dt)\otimes Pv^*v(X))$.
\end{proof}

Next we split the range of $1\otimes(1-P)$ into two pieces, namely
$$ 1\otimes(1-P)= 1\otimes v^*(1-P+\Phi_0)v(1-P)\,\oplus\,
1\otimes v^*(P-\Phi_0)v(1-P).$$

\begin{lemma}\label{lm:no-solns-2} 
The equation $(\tilde{\D}_v+V)\rho=0$ has no nonzero solutions in 
the subspace
$$L^2(\R_+,(1+t^2)dt)\otimes v^*(1-P+\Phi_0)v(1-P)(X).$$
\end{lemma}

\begin{proof}
Suppose we did have a solution $\rho\in L^2(\R_+,(1+t^2)dt)\otimes 
v^*(1-P+\Phi_0)v(1-P)(X)$.
We write $\rho(t)=\frac{1}{\sqrt{1+t^2}}\s(t)$, where $\s$ is now an (ordinary)
$L^2$ function with values in $v^*(1-P+\Phi_0)v(1-P)(X)$. A brief calculation 
shows that
\begin{align*}
\frac{1}{1+t^2}\frac{d}{dt}\la\s(t)|\s(t)\ra_X& =\frac{1}{\sqrt{1+t^2}}
\frac{d}{dt}\sqrt{1+t^2}\la\rho(t)|\rho(t)\ra_X\nno
&=\left\la-(\D+V(t))\rho(t)|\rho(t)\right\ra_X+
\left\la\rho(t)|-(\D+V(t))\rho(t)\right\ra_X.
\end{align*}
Since $v^*\D v$ is non-positive and $\D$ strictly negative on 
$v^*(1-P+\Phi_0)v(1-P)(X)$, we have the estimate
$$\D+V(t)=(1+t^2)^{-1}(t^2v^*\D v+\D)<{-c1}/(1+t^2),$$
where $c>0$ and $0<c<|r_{-1}|$ where $r_{-1}$ is the first negative 
eigenvalue of $\D$ on this subspace.
Thus
$$\frac{1}{1+t^2}\frac{d}{dt}\la\s(t)|\s(t)\ra_X\geq\frac{2c}{1+t^2}
\left\la\rho(t)|\rho(t)\right\ra_X.$$
Multiplying by $1+t^2$ and 
integrating from $0$ to $s$ gives (this is an integral of a continuous function 
into the positive cone of the $C^*$-algebra $F$)
$$\int_0^s\frac{d}{dt}\la\s(t)|\s(t)\ra_Xdt=
\la\s(s)|\s(s)\ra_X-\la\s(0)|\s(0)\ra_X\nno
\geq 2c\int_0^s \la \rho(t)|\rho(t)\ra_Xdt.
$$
The right hand side is a nondecreasing function of $s$, and if $\rho$ 
is nonzero, this function is eventually positive. 
Hence $\la\s(s)|\s(s)\ra_X$ is a continuous
non-decreasing function of $s$ in $F^+$, and so can not be integrable as can be 
seen by evaluating on a state of $F.$ Hence $\s$ is 
not an element of $L^2$ and there are no nonzero solutions $\rho$ of 
$(\tilde{\D}_v+V)(\rho)=0$ 
in the space $L^2(\R_+,(1+t^2)dt)\otimes v^*(1-P+\Phi_0)v(1-P)(X)$.
\end{proof}

Finally, we come to the subspace $L^2(\R_+,(1+t^2)dt)\otimes 
v^*(P-\Phi_0)v(1-P)(X)$. On 
this subspace we will define a parametrix which is a right inverse, 
but is not a left inverse,
instead providing solutions to our equation.
Thus,  for $t\geq s\geq 0$ define an operator $H(t,s)$ on the space $v^*(P-\Phi_0)v(1-P)(X)$
by:
$$
e^{-(t-s)v^*\D v}+\sum_{n=1}^\infty \int_s^t\int_s^{t_1}\cdots
\int_s^{t_{n-1}}((1+t_1^2)\cdots(1+t_n^2))^{-1}
e^{-(t-t_1)v^*\D v}\,v^*dv\cdots v^*dv\,e^{-(t_n-s)v^*\D v}
{\bf dt}.
$$
where $v^*\D v$ means $v^*\D v$ restricted to the subspace
$v^*(P-\Phi_0)v(1-P)(X).$

\begin{lemma}\label{lm:Second-kernel}
For all $t\geq s\geq 0$ the integrals and the 
infinite sum defining $H(t,s)$ converge absolutely in norm.
For $t\geq s\geq 0$, $H(t,s)$ is an endomorphism of the module
$v^*(P-\Phi_0)v(1-P)(X)$ with norm 
$$ \Vert H(t,s)\Vert \leq \Vert e^{-(t-s)v^*\D v}\Vert
\,e^{\tan^{-1}(t)\Vert v^*dv\Vert}
\leq e^{-(t-s)r_1}e^{\tan^{-1}(t)\Vert v^*dv\Vert},$$
where $r_1$ is the smallest positive eigenvalue of $v^*\D v$ on this
subspace. The family of endomorphisms
$H(t,s)$ satisfies the differential equations
$$\frac{d}{dt}H(t,s)=-(\D+V(t))H(t,s),\qquad \frac{d}{ds}H(t,s)=H(t,s)(\D+V(s)).$$
\end{lemma}

\begin{proof} Except for the final estimate the proof of this is similar
to the proof of Lemma \ref{lm:First-kernel}.
Now, the norm of $H(t,s)$ (on $v^*(P-\Phi_0)v(1-P)(X)$) can be estimated as 
follows:
\begin{align*} 
\Vert H(t,s)\Vert&\leq \Vert e^{-(t-s)v^*\D v}\Vert
\left( 1+\sum_{n=1}^\infty\Vert v^*dv\Vert^n
\int_0^t\int_0^{t_1}\cdots\int_0^{t_{n-1}}((1+t_1^2)\cdots
(1+t_n^2))^{-1}{\bf dt}\right)\nno
&=\Vert e^{-(t-s)v^*\D v}\Vert\left(1+\sum_{n=1}^\infty 
\frac{\Vert v^*dv\Vert^n}{n!}(\tan^{-1}(t))^n\right)\nno
&=\Vert e^{-(t-s)v^*\D v}\Vert e^{\tan^{-1}(t)\Vert v^*dv\Vert}\leq e^{-(t-s)r_1} e^{\tan^{-1}(t)\Vert v^*dv\Vert},
\end{align*}
where $r_1$ is the smallest positive eigenvalue of $v^*\D v$ on the subspace.
\end{proof}

We now define a local parametrix on the space
$L^2(\R_+,(1+t^2)dt)\otimes v^*(P-\Phi_0)v(1-P)(X).$ Let $\rho$ be given by a
continuous function in $L^2(\R_+,(1+t^2)dt)\otimes v^*(P-\Phi_0)v(1-P)(X)$
and let $t\geq 0.$ Define
$$(\tilde{R}\rho)(t):=
\frac{1}{\sqrt{1+t^2}}\int_0^tH(t,s)\sqrt{1+s^2}\rho(s)ds.$$
As in the proof of Lemma \ref{lm:no-solns-1} $\tilde{R}$ defines a closable
linear mapping locally on $[0,M]$ on it's initial dense domain of continuous 
functions. We note that $\tilde{R}(\rho)$ is differentiable.

\begin{lemma}\label{lm:solns-1} 
For every vector $x$ in the subspace $v^*(P-\Phi_0)v(1-P)(X)$ there exists a 
unique element 
$\rho\in L^2(\R_+,(1+t^2)dt)\otimes v^*(P-\Phi_0)v(1-P)(X)$ with $\rho(0)=x$ and 
$(\tilde{\D}_v+V)\rho=0$. Moreover, these are the only solutions in the space
$L^2(\R_+,(1+t^2)dt)\otimes v^*(P-\Phi_0)v(1-P)(X).$
\end{lemma}

\begin{proof}
As in the proof of Lemma \ref{lm:no-solns-1} we work locally with $t$ in the 
interval $[0,M]$, however, we suppress the local notations $\rho_M,$ etc.
Take $\rho$ a continuous function in 
$L^2(\R_+,(1+t^2)dt)\otimes v^*(P-\Phi_0)v(1-P)(X)$ with values in 
${\rm dom}(\D)$ and compute using the differential equations from Lemma \ref{lm:Second-kernel}.
$$\frac{1}{\sqrt{1+t^2}}\p_t\sqrt{1+t^2}(\tilde{R}\rho)(t)
=\frac{1}{\sqrt{1+t^2}}\p_t\left(\int_0^tH(t,s)\sqrt{1+s^2}\rho(s)ds\right)\nno
=\rho(t)-(\D+V(t))(\tilde{R}\rho)(t),$$
 Thus
$(\tilde{\D}_v+V(t))(\tilde{R}\rho)(t)=\rho(t)$ and $\tilde{R}$
is  injective.
The injectivity is first proved locally on $[0,M]$ by using
Lemma \ref{lm:operator} which easily implies global injectivity.
On the other hand if $\rho$ is smooth and lies  in the domain of 
$\tilde{\D}_v+V$ then
$(\tilde{\D}_v+V)(\rho)$ is continuous and so locally we get:
\begin{align*}
(\tilde{R}(\tilde{\D}_v+V)\rho)(t)
&=\frac{1}{\sqrt{1+t^2}}
\int_0^tH(t,s)\left(\p_s(\sqrt{1+s^2}\rho(s))+\sqrt{1+s^2}(\D+V(s)) 
\rho(s)\right)ds\nno
&=\frac{1}{\sqrt{1+t^2}}\int_0^t\p_s(H(t,s)\sqrt{1+s^2}\rho(s))ds
-\frac{1}{\sqrt{1+t^2}}\int_0^t(\p_sH(t,s))\,\sqrt{1+s^2}\rho(s)ds\nno
&\qquad\qquad\qquad+\frac{1}{\sqrt{1+t^2}}
\int_0^tH(t,s)\sqrt{1+s^2}(\D+V(s)) \rho(s)ds\nno
&=\rho(t)-\frac{1}{\sqrt{1+t^2}}H(t,0)\rho(0),
\end{align*}
where we have again used the differential equations from Lemma 
\ref{lm:Second-kernel}. 
Applying Lemma \ref{lm:operator} we obtain this equation for all
$\rho\in {\rm dom}(\tilde{\D}_v+V).$ By the estimate on 
$\Vert H(t,0)\Vert$ in the previous lemma,
the function $\frac{1}{\sqrt{1+t^2}}H(t,0)\rho(0)$ is in 
$L^2(\R_+,(1+t^2)dt)\otimes v^*(P-\Phi_0)v(1-P)(X),$
and so if $\rho$ is in the kernel of $\tilde{\D}_v+V$ we have locally and hence
globally:
\begin{equation} 
\rho(t)=(1+t^2)^{-1/2}H(t,0)\rho(0).
\label{eq:soln-formula}
\end{equation}
Conversely, with $x=\rho(0)\in v^*(P-\Phi_0)v(1-P)(X)$, 
 Eq. \eqref{eq:soln-formula} defines a solution as $\tilde{R}$ is
injective.
\end{proof}

Putting together Lemmas \ref{lm:no-solns-1}, \ref{lm:no-solns-2}, 
\ref{lm:solns-1}, we 
have the following preliminary result.

\begin{cor}\label{cr:ker} 
The kernel of $\tilde{\D}_v+V$ on $L^2(\R_+,(1+t^2)dt)\otimes v^*v(X)$
is isomorphic to the right $F$-module 
$v^*(P-\Phi_0)v(1-P)(X).$
Consequently
$${\rm{ker}}(e_v((\p_t+\D)\otimes 1_2)e_v)=
{\rm{ker}}(\widehat{e}_v((\p_t+\D)\otimes 1_2)\widehat{e}_v)
\cong{\rm{ker}}(\tilde{\D}_v+V)\cong v^*(P-\Phi_0)v(1-P)(X).$$
\end{cor}

Thus we have part of the Index of $(e_v(\hat{\D}\otimes 1_2)e_v).$ To 
complete the 
calculation, we compute the kernel of the adjoint operator $e_v(-\p_t+\D)e_v$.
We follow an essentially similar path, 
but must take a little more care with the extended $L^2$-space $\hat{\E}.$

\subsection{The kernel of the adjoint}


As explained above, we must compute the kernel of the operator
$e_v\bma -\p_t +\D & 0\\ 0 & -\p_t +\D \ema e_v$
as a map from 
$e_v\bma \hat{\E}\\ \hat{\E}\ema$ to $e_v\bma\E\\ \E\ema.$
Recall that $M(F,A)$ acts as {\bf zero} on the constant
$X_0$-valued functions but the added unit element acts as the identity. Thus
for a pair of constant functions
$\bma x_1\\x_2\ema\in\bma \hat{\E}\\ \hat{\E}\ema$
we have $e_v\bma x_1\\x_2\ema=
\bma x_1\\0\ema.$
Hence $e_v\bma \hat{\E}\\ \hat{\E}\ema\subseteq
\bma \hat{\E}\\ \E\ema.$
For $\xi\in e_v\bma \hat{\E}\\\hat{\E}\ema$ to be in the domain of 
$e_v((-\p_t+\D)\otimes 1_2)e_v$ we impose the boundary condition:
$$\bma (1-vv^*)(1-P) & 0 \\ 0 & v^*v(1-P)\ema\xi(0) = 0.$$
For the constant function 
$\bma x_1\\0\ema\in e_v\bma \hat{\E}\\ \hat{\E}\ema$ to be in the domain 
this means that $x_1$ must satisfy $(1-vv^*)(1-P)(x_1)=0.$ However, this is
automatic as $x_1\in X_0$ so that $(1-P)(x_1)=0.$ Thus the domain of 
 $e_v((-\p_t+\D)\otimes 1_2)e_v$ extended to the constant $X_0$-valued
functions $(X_0\oplus X_0)^T$
is $(X_0\oplus 0)^T.$ Of course, the extended operator 
$e_v((-\p_t+\D)\otimes 1_2)e_v$ is identically $0$ here.
{\bf It is important to note that: ${\rm dom}(e_v((-\p_t+\D)\otimes 1_2)e_v)
\subseteq e_v(\hat{\E}\oplus\E)^T.$}

As before
we use the orthogonal decomposition
of $e_v$ to enable separate analysis of the two
subspaces:
$$\widehat{e}_v\bma\hat{\E}\\ \hat{\E}\ema\to\widehat{e}_v\bma\E\\ \E\ema
\;\;\mbox{and}\;\; 
e^0_v\bma \hat{\E}\\ \hat{\E}\ema \to e^0_v\bma\E\\ \E\ema.$$
Now, 
$$e^0_v\bma \hat{\E}\\ \hat{\E}\ema=e^0_v\bma \E\\\ \E\ema \oplus
\bma (1-vv^*)(X_0)\\ 0\ema.$$
As in the case of $e_v((\p_t+\D)\otimes 1_2)e_v$ we have 
$e_v((-\p_t+\D)\otimes 1_2)e_v$ is one-to-one on $e_v^0(\E\oplus\E)^T$
and so the kernel there is $0.$ Since 
$e_v((-\p_t+\D)\otimes 1_2)e_v$ is identically $0$ on 
$e_v^0((1-vv^*)(X_0)\oplus 0)^T\cong (1-vv^*)(X_0),$ we have the following 
result.

\begin{prop}\label{pr:first-extended-solns}The kernel
of $e_v((-\p_t+\D)\otimes 1_2)e_v$ restricted to 
$e^0_v\bma \hat{\E}\\ \hat{\E}\ema$ is isomorphic to 
the right  $F$-module $(1-vv^*)(X_0).$
\end{prop}

These solutions are a rather trivial type of {\bf extended solution} to the 
adjoint equation. Next: 
$$\widehat{e_v}\bma (-\p_t+\D) & 0 \\ 0 & (-\p_t+\D)\ema\widehat{e_v}\bma \xi_1 \\
\xi_2\ema =\bma -itv(-\p_t+\D)\xi_2 +\frac{it^2}{1+t^2}v\xi_2 -\frac{it^3}{1+t^2}
dv\xi_2 \\ (-\p_t+\D)\xi_2 -\frac{t}{1+t^2}\xi_2 +\frac{t^2}{1+t^2}v^*dv\xi_2
\ema.$$
That is, any vector $(\rho_1,\rho_2)^T$ in the range of
$\widehat{e_v}[(\p_t+\D)\otimes 1_2]\widehat{e_v}$ satisfies $\rho_1(t)=-itv(\rho_2)(t)$ and as before, after simplifying,
$$\rho_2(t)=\left(\frac{-1}{\sqrt{1+t^2}}\p_t\circ\sqrt{1+t^2}
+v^*v\D  + \frac{t^2v^*dv}{1+t^2}\right)\xi_2
=:(\widehat{\D}_v + V)\xi_2$$
$$\mbox{where}\;\;\;\widehat{\D}_v=\left( 
\frac{-1}{\sqrt{1+t^2}}\p_t\circ\sqrt{1+t^2}+v^*v\D \right)
\;\;\;\mbox{and}\;\;\; V=\frac{t^2}{1+t^2}\otimes (v^*dv):=V_0\otimes (v^*dv).$$
So in order to compute the kernel of 
$\widehat{e_v}[(-\p_t+\D)\otimes 1_2]\widehat{e_v}$ acting on the range of
$\widehat{e_v}$, it suffices to compute the kernel of $\widehat{\D}_v + V$ 
acting
on vectors $\xi_2\in \E$ 
satisfying $v^*v(\xi_2)=\xi_2$
and $-itv(\xi_2)\in \hat{\E}$. {\bf As opposed to the $T_+$ case,}
such vectors $\xi_2$ need only lie in the larger space
$L^2(\R_+)\otimes v^*v(X),$ while $\xi_1(t)=-itv(\xi_2(t))$ {\bf may} have a 
nonzero limit at $\infty$ in $X_0$ 
subject to the boundary conditions 
$P(\xi_2(0))=\xi_2(0).$

Again we split $L^2(\R_+)\otimes v^*v(X)$ into the range of $1\otimes P$ and
$1\otimes (1-P)$. 
On the image of $1\otimes (1-P)$ we define a two parameter family of 
bounded operators which will be the integral kernel of a local parametrix for 
$\widehat{\D}_v+V$ on this space. 
Thus with  $\D$ standing for $(1-P)\D$ and
for $t\geq s\geq 0$, define an operator on $(1-P)v^*v(X)$ by
$$ W(t,s)=e^{(t-s)\D}+\sum_{n=1}^\infty (-1)^n
\int_s^t\int_s^{t_1}\cdots\int_s^{t_{n-1}} 
e^{(t-t_1)\D}V(t_1)e^{(t_1-t_2)\D}V(t_2)\cdots V(t_n)e^{(t_n-s)\D}{\bf dt}.
$$

\begin{lemma}\label{lm:First-adjoint-kernel} 
For all $t\geq s\geq 0$ the integrals and the 
infinite sum defining $W(t,s)$ converge absolutely in norm.
For all $t\geq s\geq 0$ we have (in the operator norm for endomorphisms 
of $v^*v(X)$) 
$$\Vert W(t,s)\Vert\leq \Vert e^{(t-s)\D}\Vert\, e^{(t-s)\Vert v^*dv\Vert}.$$
Moreover $W(t,s)$ satisfies the differential equations
$$ \frac{d}{dt}W(t,s)=(\D+V(t))W(t,s),\qquad \frac{d}{ds}W(t,s)=-W(t,s)(\D+V(s)).$$
\end{lemma}

\begin{proof} This is very similar to the proof of Lemma \ref{lm:First-kernel}
so we omit the details.
\end{proof}

Using these results we construct a local parametrix for 
$(\widehat{\D}_v+V)(1\otimes(1-P)).$
For $\rho$ a continuous function in 
$L^2(\R_+)\otimes (1-P)v^*v(X)$ define
$$(\widehat{Q}\rho)(t):=
-(1+t^2)^{-1/2}\int_0^tW(t,s)\sqrt{1+s^2}\rho(s)ds.$$
Observe that $(\widehat{Q}\rho)(0)=0$, and is differentiable, and so if
$\rho$ has range in ${\rm dom}(\D)$ then
$\widehat{Q}(\rho)$ is locally in 
the domain of $\widehat{\D}_v+V$. As in the proof of Lemma \ref{lm:no-solns-1},
$\widehat{Q}$ defines a closable linear mapping locally on $[0,M]$ on its
initial dense domain of continuous functions. All our calculations below are 
local as in Lemma \ref{lm:no-solns-1}.

\begin{lemma}\label{lm:no-adj-solns-1} 
In the space
$L^2(\R_+)\otimes (1-P)v^*v(X)$
the equation $(\widehat{\D}_v+V)\rho=0$ has no nonzero solutions and therefore it has no nonzero solutions in the subspace
$L^2(\R_+,(1+t^2)dt)\otimes (1-P)v^*v(X).$
\end{lemma}

\begin{proof}
Let $\rho$ be a smooth function in the domain of $(\widehat{\D}_v+V)(1-P)$:
\begin{align*}  (\widehat{Q}(\widehat{\D}_v+V)\rho)(t)
&=\frac{-1}{\sqrt{1+t^2}}\int_0^tW(t,s)\left(-\p_s(\sqrt{1+s^2}\rho(s))+
\sqrt{1+s^2}(\D+V(s))\rho(s)\right)ds\nno
&=\frac{1}{\sqrt{1+t^2}}\int_0^t\p_s(W(t,s)\sqrt{1+s^2}\rho(s))ds 
-\frac{1}{\sqrt{1+t^2}}\int_0^t(\p_s W(t,s)) \sqrt{1+s^2}\rho(s)ds\nno
&\qquad\qquad\qquad\qquad\qquad
-\frac{1}{\sqrt{1+t^2}}\int_0^tW(t,s)\sqrt{1+s^2}(\D+V(s))\rho(s)ds\nno
&=\rho(t)-(1+t^2)^{-1/2}W(t,0)\rho(0) =\rho(t),
\end{align*}
where, as $\rho$ has values in the range of $(1-P)$, we have $\rho(0)=0.$
Arguing
as in the proof of Lemma \ref{lm:no-solns-1} this implies that
$(\widehat{\D}_v+V)(1-P)$ is injective on its whole domain. Hence,
$(\widehat{\D}_v+V)\rho=0$ has no nonzero solutions in 
$L^2(\R_+)\otimes (1-P)v^*v(X)$.
\end{proof}

Next we split the range of $1\otimes P$ into three pieces, namely
$$ 1\otimes P= [1\otimes v^*(P-\Phi_0)vP]\,\oplus\,
[1\otimes v^*(1-P)vP]\,\oplus\,
[1\otimes v^*\Phi_0vP].$$

\begin{lemma}\label{lm:no-adj-solns-2} In
the subspace
$L^2(\R_+)\otimes v^*(P-\Phi_0)vPv^*v(X)$
the equation $(\widehat{\D}_v+V)\rho=0$ has no nonzero solutions and therefore has no nonzero solutions in $L^2(\R_+,(1+t^2)dt)\otimes v^*(P-\Phi_0)vPv^*v(X).$
\end{lemma}

\begin{proof}
First, suppose we have a solution $\rho$ with 
$\rho(t)\in v^*(P-\Phi_0)vPv^*v(X)$ 
for all $t\geq 0$, and $\rho\in L^2(\R_+)\otimes v^*(P-\Phi_0)vPv^*v(X)$.
Then $-itv(\rho(t))\in (P-\Phi_0)vPv^*v(X)$ and so if this has a limit at
$\infty$ in $\Phi_0(X)$, the limit must be $0.$ That is,
$-itv(\rho)\in L^2(\R_+)\otimes (P-\Phi_0)vPv^*v(X)$ and so our solution
$\rho$ actually lies in the smaller space: $$L^2(\R_+,(1+t^2)dt)\otimes v^*(P-\Phi_0)vPv^*v(X).$$
Arguing as in Lemma \ref{lm:no-solns-2}
write $\rho(t)=(1+t^2)^{-1/2}\s(t)$, where $\s$ is now an (ordinary)
$L^2$ function with values in $v^*(P-\Phi_0)vPv^*v(X)$:
\begin{align*}
(1+t^2)^{-1}\frac{d}{dt}\la\s(t)|\s(t)\ra_X
\frac{d}{dt}\sqrt{1+t^2}\la\rho(t)|\rho(t)\ra_X
=\left\la(\D+V(t))\rho(t)|\rho(t)\right\ra_X+
\left\la\rho(t)|(\D+V(t))\rho(t)\right\ra_X.
\end{align*}
Since $v^*\D v$ is strictly positive and $\D$ is non-negative on 
$v^*(P-\Phi_0)vPv^*v(X)$, we have the estimate
$$\D+V(t)=(1+t^2)^{-1}(t^2v^*\D v+\D)>{r_1 t^2 /}(1+t^2),$$
where $r_1$ is the first positive 
eigenvalue of $v^*\D v$ on this subspace and therefore
$$\frac{1}{1+t^2}\frac{d}{dt}\la\s(t)|\s(t)\ra_X\geq\frac{2r_1 t^2}{1+t^2}
\left\la\rho(t)|\rho(t)\right\ra_X.$$
Multiplying by $1+t^2$ and 
integrating from $0$ to $s$ gives 
$$ \int_0^s\frac{d}{dt}\la\s(t)|\s(t)\ra_Xdt
=
\la\s(s)|\s(s)\ra_X-\la\s(0)|\s(0)\ra_X\nno
\geq 2r_1\int_0^s t^2 \la \rho(t)|\rho(t)\ra_Xdt.$$
The right hand side is a nondecreasing function of $s$, and if $\rho$ 
is nonzero, this function is eventually positive. Thus arguing further as in Lemma \ref{lm:no-solns-2}
there are no nonzero solutions $\rho$ of 
$(\widehat{\D}_v+V)(\rho)=0$ 
in  $L^2(\R_+,(1+t^2)dt)\otimes v^*(P-\Phi_0)vPv^*v(X),$ and hence
none in $L^2(\R_+)\otimes v^*(P-\Phi_0)vPv^*v(X).$
\end{proof}

Next, we come to the subspace $L^2(\R_+)\otimes v^*(1-P)vPv^*v(X)$. On 
this subspace we will define a local parametrix which is a right inverse, 
but is not a left inverse,
instead providing solutions to our equation.
So for $t\geq s\geq 0$ define $G(t,s)$ (on the module $v^*(1-P)vPv^*v(X)$)
by
$$
 e^{(t-s)v^*\D v}+\sum_{n=1}^\infty (-1)^n\int_s^t\int_s^{t_1}\cdots
 \int_s^{t_{n-1}} ((1+t_1^2)\cdots(1+t_n^2))^{-1}\,
e^{(t-t_1)v^*\D v}\,v^*dv\cdots v^*dv\,e^{(t_n-s)v^*\D v}
{\bf dt}.
$$

\begin{lemma}\label{lm:Second-adj-kernel}
For all $t\geq s\geq 0$ the integrals and the 
infinite sum defining $G(t,s)$ converge absolutely in norm.
For $t\geq s\geq 0$, $G(t,s)$ is a bounded endomorphism of the module
$v^*(1-P)vPv^*v(X)$ with norm bounded by
$$ \Vert G(t,s)\Vert \leq 
\Vert e^{(t-s)v^*\D v}\Vert\,e^{(t-s)\Vert v^*dv\Vert}
\leq e^{(t-s)r_{-1}}e^{(t-s)\Vert v^*dv\Vert},$$
where $r_{-1}$ is the largest negative eigenvalue of $\D$. The family of 
endomorphisms
$G(t,s)$ satisfies the differential equations
$$\frac{d}{dt}G(t,s)=(\D+V(t))G(t,s),\qquad \frac{d}{ds}G(t,s)=-G(t,s)(\D+V(s)).$$
\end{lemma}

\begin{proof}
The proof of this is very similar to the proof of Lemma \ref{lm:Second-kernel}.
\end{proof}

Now define a local parametrix on continuous functions $\rho$ by
$$(\widehat{R}\rho)(t):=
(1+t^2)^{-1/2}\int_0^tG(t,s)\sqrt{1+s^2}\rho(s)ds,\quad \rho\in 
L^2(\R_+)\otimes v^*(1-P)vPv^*v(X).$$
As in the proof of Lemma \ref{lm:no-solns-1} $\widehat{R}$ defines a closable
linear mapping locally on $[0,M]$ on the initial dense domain of continuous 
functions. We note that $\widehat{R}(\rho)$ is differentiable.

\begin{lemma}\label{lm:adj-solns-1} 
For every vector $x$ in the space $v^*(1-P)vPv^*v(X)$ there exists a unique 
element
$\rho\in L^2(\R_+)\otimes v^*(1-P)vPv^*v(X)$ with $\rho(0)=x$ and 
$(\widehat{\D}_v+V)\rho=0$. Moreover, these are the only solutions in the
subspace $L^2(\R_+)\otimes v^*(1-P)vPv^*v(X).$ In fact these solutions $\rho$
clearly lie in $L^2(\R_+,(1+t^2)dt)\otimes v^*(1-P)vPv^*v(X)$ and satisfy
$\lim_{t\to\infty} -itv(\rho(t)) =0.$ 
\end{lemma}

\begin{proof}
We work locally as in the proofs of Lemmas \ref{lm:no-solns-1} 
and \ref{lm:solns-1}.
Take $\rho$ a continuous function in $L^2(\R_+)\otimes v^*(1-P)vPv^*v(X)$ 
with values in ${\rm dom}(\D)$ and compute
$$\frac{-1}{\sqrt{1+t^2}}\p_t\sqrt{1+t^2}(\widehat{R}\rho)(t)
=\frac{1}{\sqrt{1+t^2}}\p_t\left(\int_0^tG(t,s)\sqrt{1+s^2}\rho(s)ds\right)
=\rho(t)-(\D+V(t))(\widehat{R}\rho)(t),$$
where we have used the computations from 
Lemma \ref{lm:Second-adj-kernel}. Thus
$(\widehat{\D}_v+V(t))(\widehat{R}\rho)(t)=\rho(t)$
and
$\widehat{R}$ is injective.
The injectivity is first proved locally on $[0,M]$ by using
Lemma \ref{lm:operator} which easily implies global injectivity.
On the other hand if $\rho$ is smooth and 
lies in the domain of $\widehat{\D}_v+V$ then
$(\widehat{\D}_v+V)(\rho)$ is continuous and so locally we get:
\begin{align*}
(\widehat{R}(\widehat{\D}_v+V)\rho)(t)
&=\frac{-1}{\sqrt{1+t^2}}
\int_0^tG(t,s)\left(-\p_s(\sqrt{1+s^2}\rho(s))+\sqrt{1+s^2}(\D+V(s)) 
\rho(s)\right)ds\nno
&=\frac{1}{\sqrt{1+t^2}}\int_0^t\p_s(G(t,s)\sqrt{1+s^2}\rho(s))ds
-\frac{1}{\sqrt{1+t^2}}\int_0^t(\p_sG(t,s))\,\sqrt{1+s^2}\rho(s)ds\nno
&\qquad\qquad\qquad-\frac{1}{\sqrt{1+t^2}}
\int_0^tG(t,s)\sqrt{1+s^2}(\D+V(s)) \rho(s)ds\nno
&=\rho(t)-(1+t^2)^{-1/2}G(t,0)\rho(0),
\end{align*}
where we have again used the derivative computations from Lemma 
\ref{lm:Second-adj-kernel}. Applying Lemma \ref{lm:operator} we get this
formula for all $\rho\in {\rm dom}(\widehat{\D}_v+V).$

Now if $\rho$ is in the kernel of $\widehat{\D}_v+V$ we have locally and 
hence globally
\begin{equation} 
\rho(t)=(1+t^2)^{-1/2}G(t,0)\rho(0),
\label{eq:adj-soln-formula}
\end{equation}
and this lies in $L^2(\R_+)\otimes v^*(1-P)vPv^*v(X)$ by the estimate:
$$\Vert G(t,0)\Vert\leq e^{tr_{-1}} e^{\tan^{-1}(t)\Vert v^*dv\Vert}$$
where $r_{-1}$ is the largest negative eigenvalue of $\D$ on the subspace.
Conversely, given any vector $\rho(0)\in v^*(1-P)vPv^*v(X)$, Equation 
\eqref{eq:adj-soln-formula}
defines a solution since $\widehat{R}$ is injective.
\end{proof}

Finally we need to consider the subspace $L^2(\R_+)\otimes v^*\Phi_0vP(X)$.
This subspace gives rise to extended solutions. That is, the solutions we 
seek here are the second components $\xi_2$ of a solution 
$\xi=(\xi_1,\xi_2)^T$ in $e_v(\hat{\E}\oplus \hat{\E})^T\subseteq 
(\hat{\E}\oplus\E)^T$ to the equation $e_v(-\p_t+\D)e_v\xi=0$,
where $\xi_1\in\hat{\E}$ satisfies $\xi_1=-ivt\xi_2$. Hence, a true extended 
solution (one where $\xi_1 \notin\E$) comes from those $\xi_2$ which 
behave like $(1+t^2)^{-1/2}$ as $t\to\infty$. With this reminder, we have

\begin{lemma}\label{lm:adj-solns-2} 
For every vector $x\in v^*(\Phi_0)vP(X)$ there exists a unique solution
to the equation $(\widehat{\D}_v+V)\rho=0$ in the space
$L^2(\R_+)\otimes v^*(\Phi_0)vP(X)$ with $\rho(0)=x.$ Moreover, every solution
in this space is of the form
$ \rho(t)=(1+t^2)^{-1/2}e^{-(\tan^{-1}(t))v^*dv}\rho(0)$ and 
\begin{align*}
&{\rm (1)} \lim_{t\to\infty} -itv(\rho(t))=-ive^{-\pi/2 v^*dv}(\rho(0))\in 
\Phi_0(X)\;\;\;{\rm and}\nno
&{\rm (2)}\;\; t\mapsto \left(-ivt\rho(t)+ive^{-(\pi/2)v^*dv}(\rho(0)\right)
\;\;\; {\rm is\;\;\; in}\;\;\; 
L^2(\R_+)\otimes v^*(\Phi_0)vP(X).
\end{align*}
\end{lemma}

\begin{proof}
We define a local parametrix for $\rho$ a continuous function in
$L^2(\R_+)\otimes v^*(\Phi_0)vP(X)$ by
$$ (\widehat{E}\rho)(t):=
\frac{-1}{\sqrt{1+t^2}}e^{-(\tan^{-1}(t))v^*dv}\int_0^t e^{(\tan^{-1}(s))v^*dv}
\sqrt{1+s^2}\rho(s)ds.$$
We observe that $(\widehat{E}\rho)$ is differentiable and satisfies
$(\widehat{E}\rho)(0)=0.$  
To show that this a parametrix, first use $v^*dv=v^*\D v-v^*v\D$ to rewrite
$$\widehat{\D}_v+V=\frac{-1}{\sqrt{1+t^2}}\p_t\sqrt{1+t^2}+
v^*\D v-\frac{1}{1+t^2}v^*dv.$$
As $v^*\D v$ acts as zero on $v^*\Phi_0v(X)$, this 
reduces on $v^*\Phi_0 v(X)$ to
$$\widehat{\D}_v+V=
\frac{-1}{\sqrt{1+t^2}}\p_t\sqrt{1+t^2}-\frac{1}{1+t^2}v^*dv.$$

Applying $\frac{-1}{\sqrt{1+t^2}}\p_t\sqrt{1+t^2}$ to 
$\widehat{E}(\rho)$ and using the product rule gives
\begin{align*}
\frac{-1}{\sqrt{1+t^2}}\p_t\sqrt{1+t^2}(\widehat{E}\rho)(t)
&=
\frac{1}{\sqrt{1+t^2}}
\p_t\left(e^{-(\tan^{-1}(t))v^*dv}\int_0^t e^{(\tan^{-1}(s))v^*dv}
\sqrt{1+s^2}\rho(s)ds.\right)\nno
&=\rho(t)-\frac{v^*dv}{1+t^2}(\widehat{E}\rho)(t).
\end{align*}
Thus $(\widehat{\D}_v+V)(\widehat{E}\rho)=\rho$ locally for continuous functions
As in previous cases $\widehat{E}$ is locally a closable operator
and so by Lemma \ref{lm:operator} we get that 
$(\widehat{\D}_v+V)(\widehat{E}\rho)=\rho$ locally for all $\rho$ in the 
domain of $\widehat{E}$. Hence, $\widehat{E}$ is globally injective.  
Integration by parts for smooth $\rho$ in the domain gives
$$(\widehat{E}(\widehat{\D}_v+V)\rho)(t)=
\rho(t)-(1+t^2)^{-1/2}e^{-\tan^{-1}(t)v^*dv}\rho(0).$$
Applying Lemma \ref{lm:operator}, we get this equation for all 
$\rho\in {\rm dom}(\widehat{\D}_v+V).$
Hence if $\rho\in {\rm ker}(\widehat{\D}_v+V)$ we have
$$\rho(t)=(1+t^2)^{-1/2}e^{-\tan^{-1}(t)v^*dv}\rho(0).$$
On the other hand if $x\in v^*(\Phi_0)vP(X)$ and we define $\rho$ by this 
equation with $\rho(0)=x$ then we have a solution of 
$(\widehat{\D}_v+V)(\rho)=0$ in the space 
$\rho\in L^2(\R_+)\otimes v^*(\Phi_0)vP(X).$
Since for each $t\geq 0$ we have $\rho(t)\in v^*\Phi_0 v(X)$, we also have 
$-itv(\rho(t))\in\Phi_0 v(X)\subseteq \Phi_0(X),$ and therefore:
$$\lim_{t\to\infty} -itv(\rho(t))=-ive^{-\pi/2 v^*dv}(\rho(0))\in \Phi_0(X).$$
It is an exercise to check that $t\mapsto 
\left(-ivt\rho(t)+ive^{-(\pi/2)v^*dv}(\rho(0)\right)$ 
is in $L^2(\R_+)\otimes v^*(\Phi_0)vP(X).$
\end{proof}

Putting together Proposition \ref{pr:first-extended-solns} and
Lemmas \ref{lm:no-adj-solns-1}, \ref{lm:no-adj-solns-2}, 
\ref{lm:adj-solns-1}, \ref{lm:adj-solns-2}, we 
have the following.

\begin{cor}\label{cr:adj-ker} 
The kernel of $(\widehat{\D}_v+V)$ on $L^2(\R_+)\otimes v^*v(X)$
is isomorphic to the right $F$-module 
$$\ker(\widehat{\D}_v+V)\cong [v^*(1-P)vP(X)]\oplus [v^*\Phi_0 vP(X)],$$
where the first summand consists of ordinary solutions in  
$L^2(\R_+,(1+t^2)dt)\otimes v^*v(X),$ while the second summand consists of 
extended solutions whose second component is in $L^2(\R_+)\otimes v^*v(X).$ 
Consequently, taking into account the (trivial) extended solutions
of Proposition \ref{pr:first-extended-solns}, $(1-vv^*)\Phi_0(X)$
we have the full kernel
$${\rm{ker}}(e_v((-\p_t+\D)\otimes 1_2)e_v)
\cong [v^*(1-P)vP(X)]\oplus [v^*\Phi_0 vP(X)]\oplus [(1-vv^*)\Phi_0(X)].$$
\end{cor}

\subsection{Completing the proof of Theorem \ref{mainresult}}

Consider the pairing of $\bma 1 & 0\\ 0 &0\ema$ with
$\p_t+\D$. Examining our earlier parametrix computations shows that $\p_t+\D$
with boundary
condition $P$ has no kernel, while $-\p_t+\D$ with
boundary condition
$1-P$ has  extended solutions: the constant functions with
value in $X_0$. The projection onto these extended solutions
is $\Phi_0$ and
$\mbox{Index}(\p_t+\D)=-[X_0].$
Since the mapping cone algebra is nonunital, we can not just pair with
the class of $e_v$, but must pair with $[e_v]-[\bma 1 & 0\\ 0 &
0\ema]$. We have computed the pairing of $(\hat{X},\hat\D)$ with both these
terms, and so we have the following intermediate result:

\begin{prop} The pairing of $[e_v]-[\bma 1 & 0\\ 0 &0\ema]$ with
$(\hat{X},\hat\D)$ is given by
\bean&&{\rm Index}(e_v(\p_t+\D)e_v)-{\rm Index}(\p_t+\D)=
{\rm Index}(e_v(\p_t+\D)e_v)+[X_0]\nno
&=&
[v^*(P-\Phi_0)v(1-P)(X)]
-[v^*(1-P)vP(X)]-[v^*\Phi_0vP(X)]-[(1-vv^*)(X_0)]+[X_0]\nno
&=&[v^*Pv(1-P)(X)]-[v^*\Phi_0v(1-P)(X)]
-[v^*(1-P)vP(X)]-[v^*\Phi_0vP(X)]+[vv^*(X)_0]\nno
&=&[v^*Pv(1-P)(X)]
-[v^*(1-P)vP(X)]-[v^*\Phi_0 v(X)]+[vv^*(X_0)]\nno
&=&
[v^*Pv(1-P)(X)]
-[v^*(1-P)vP(X)].\eean
\end{prop}
The last line follows because
$w=v^*\Phi_0$ is a partial isometry with
$ww^*=v^*\Phi_0v$ and $ w^*w=vv^*\Phi_0,$
showing that the modules defined by these projections are isomorphic.

Now we can finalise the proof of the Theorem by computing the index of
$$PvP:v^*vP(X)\to vv^*P(X),$$
where $P$ is the non-negative spectral projection for $\D$.
The kernel of $PvP$ is given by the set
$$\{\xi\in v^*vP(X): v\xi\in vv^*(1-P)(X)=(1-P)v(X)\}=Pv^*(1-P)v(X),$$
while the cokernel is given by
$$\{\xi\in vv^*P(X)=Pv(X):\xi=v\eta,\ \ \eta\in v^*v(1-P)(X)\}=(1-P)v^*Pv(X).$$
Thus
$$\mbox{Index}(PvP)=[Pv^*(1-P)v(X)]-[(1-P)v^*Pv(X)]\in K_0(F).$$

Hence
$$\mbox{Index}(PvP:v^*vP(X)\to
vv^*P(X))=-\left({\rm Index}(e_v(\p_t+\D)e_v)-{\rm
Index}(\p_t+\D)\right),$$
and the proof of Theorem \ref{mainresult} is complete.


\bigskip
\centerline{***********************************}

{\bf Remark} When $[\D,v^*dv]=0$, enormous 
simplifications occur in the preceeding 
analysis. In this case one can verify that for the 
equation $\tilde{\D}_v+V$
in $v^*v\E$, a solution
of
$\rho=(\tilde{\D}_{v}+V)\xi$ vanishing at zero is given by
$$\xi(t)=
\frac{e^{v^*dv\tan^{-1}(t)}}{\sqrt{1+t^2}}\int_0^t
e^{-v^*\D v(t-s)}\sqrt{1+s^2} e^{-v^*dv\tan^{-1}(s)}\rho(s)ds,$$
and we require $\rho\in v^*(P-\Phi_0)v\E$. This formula can be obtained by performing the sums and integrals in the definition of our 
more general parametrix. Similar comments apply to the other cases.

In the next section we apply Theorem 5.1 to graph algebras and the Kasparov 
module constructed from the gauge action in \cite{pr}. We will see that in 
this case we can 
always assume that $v^*dv$ commutes with $\D$, so that 
we are in the 
simplest situation described above.

\section{Applications to certain Cuntz-Krieger systems}

For a  detailed introduction to Cuntz-Krieger systems as graph algebras
see \cite{R}.
A directed graph $E=(E^0,E^1,r,s)$ consists of
countable sets $E^0$ of vertices and $E^1$ of edges, and maps
$r,s:E^1\to E^0$ identifying the range and source of each edge.
{\bf We 
will always assume that the graph is} {\bf locally-finite}
which means that each vertex emits at most finitely many edges
and each vertex receives at most
finitely many edges. We write $E^n$ for the set of paths
$\mu=\mu_1\mu_2\cdots\mu_n$ of length $|\mu|:=n$; that is,
sequences of edges $\mu_i$ such that $r(\mu_i)=s(\mu_{i+1})$ for
$1\leq i<n$.  The maps $r,s$ extend to $E^*:=\bigcup_{n\ge 0} E^n$
in an obvious way. 
 A {\bf sink}
is a vertex $v \in E^0$ with $s^{-1} (v) = \emptyset$, a
{\bf source} is a vertex $w \in E^0$ with $r^{-1} (w) =
\emptyset$ however we will always assume there are {\bf no sources}.

A {\bf Cuntz-Krieger $E$-family} in a $C^*$-algebra $B$ consists
of mutually orthogonal projections $\{p_v:v\in E^0\}$ and partial
isometries $\{S_e:e\in E^1\}$ satisfying the {\bf Cuntz-Krieger
relations}
\begin{equation*}
S_e^* S_e=p_{r(e)} \mbox{ for $e\in E^1$} \ \mbox{ and }\
p_v=\sum_{\{ e : s(e)=v\}} S_e S_e^*   \mbox{ whenever $v$ is not
a sink.}
\end{equation*}
There is a universal 
$C^*$-algebra $C^*(E)$ generated by a non-zero Cuntz-Krieger
$E$-family $\{S_e,p_v\}$ \cite[Theorem 1.2]{kpr} .  A product
$S_\mu:=S_{\mu_1}S_{\mu_2}\dots S_{\mu_n}$ is non-zero precisely
when $\mu=\mu_1\mu_2\cdots\mu_n$ is a path in $E^n$. The
Cuntz-Krieger relations imply that 
words in $\{S_e,S_f^*\}$ collapse to products of the form $S_\mu
S_\nu^*$ for $\mu,\nu\in E^*$ satisfying $r(\mu)=r(\nu)$
and we have 
\begin{equation}
C^*(E)=\clsp\{S_\mu S_\nu^*:\mu,\nu\in E^*\mbox{ and
}r(\mu)=r(\nu)\}.\label{spanningset}
\end{equation}
There is a canonical gauge action of $\bf T$ on $A:=C^*(E)$ 
determined on the
generators via: $\gamma_z(p_v)=p_v$ and $\gamma_z(S_e)=zS_e.$  
Because ${\bf T}$ is compact,
averaging over $\gamma$ with respect to normalised Haar measure
gives a faithful expectation $\Phi$ from $A$ onto the fixed-point
algebra $F=A^\gamma$:
\[
\Phi(a):=\frac{1}{2\pi}\int_{\bf T} \gamma_z(a)\,d\theta\ \mbox{ for
}\ a\in C^*(E),\ \ z=e^{i\theta}.
\]

As described in \cite{pr}, 
right multiplication by $F$ makes $A$ into a 
right (pre-Hilbert) $F$-module 
with inner product: $(a|b)_R:=\Phi(a^*b).$ Then $X$ denotes the Hilbert 
$F$-module completion of $A$ in the norm 
$$\Vert a\Vert_X^2:=\Vert(a|a)_R\Vert_F=
\Vert\Phi(a^*a)\Vert_F.$$
For each $k\in {\bf Z}$, the projection $\Phi_k$ onto the $k$-th 
spectral
subspace of the gauge action is defined by
$$\Phi_k(x)=\frac{1}{2\pi}\int_{\bf T}z^{-k}\gamma_z(x)d\theta,
\ \ z=e^{i\theta},\ \ x\in X.$$
The generator of the gauge action on $X$,
$\D=\sum_{k\in\Z}k\Phi_k$, is determined 
on the generators of $A=C^*(E)$ by the formula
 $$\D(S_\al S_\beta^*)=(|\al|-|\beta|)S_\al S_\beta^*.$$

 The following result is proved in \cite{pr}.

\begin{prop}\label{Kasmodule}
Let $A$ be the graph $C^*$-algebra 
of a directed graph with no sources. Then 
$(X,\D)$ is an odd unbounded Kasparov $A$-$F$-module.
The operator $\D$ has discrete spectrum, and commutes with 
left multiplication by $F\subset A$. 
Set  $V=\D(1+\D^2)^{-1/2}$. Then $(X,V)$
defines a class in $KK^1(A,F)$.
\end{prop}


We are going to investigate relations in $K_0(M(F,A))$. 
As graph algebras are generated by partial
isometries in $A$ with range and source in $F$, so $K_0(M(F,A))$
contains a lot of information about $A$ and the underlying
graph. The main result of Section 5 will give us more information.

\begin{prop}\label{gens} Let $A$ be the graph $C^*$-algebra of a locally finite
directed graph. Let $\al=\al_1\al_2\cdots\al_{|\al|}$ be a path in the
graph, and $S_\al$ the
corresponding partial isometry in $A$. If $\mu$ is also
a path let $P_\mu=S_\mu S_\mu^*$.
Then in $K_0(M(F,A))$ we have
the relations
$$ [S_\al P_\mu]=\sum_{j=1}^{|\al|-1}[S_{\al_j}S_{\al_{j+1}}
S_{\al_{j+2}}\cdots S_{\al_n}P_\mu S_{\al_n}^*
\cdots S_{\al_{j+2}}^*S_{\al_{j+1}}^*]\
+[S_{\al_{|\al|}}P_\mu],$$
$$[S_\al S_\beta^*]=[S_\al]-[S_\beta],\quad \al,\,\beta\ \ {\rm paths}.$$
\end{prop}

\begin{proof}  
This proceeds by induction on $|\al|$. If $|\al|=0$ then
$[S_\al]=[p_{r(\al)}]=0$ and if $|\al|=1$, there is nothing to
prove. So suppose the relation is true for all $\al$ with
$|\al|<n$.  Let $\al$ be a path with $|\al|=n$ and write $\al=\ual\al_n$
where $|\ual|=n-1$.
 Then
\begin{align*} [S_\al P_\mu]&=[S_{\ual} S_{\al_n}P_\mu]
=[S_{\ual} S_{\al_n}P_\mu S_{\al_n}^* S_{\al_n}P_\mu]
=[S_{\ual} S_{\al_n}P_\mu S_{\al_n}^*]+[S_{\al_n}P_\mu]\ \ \mbox{by Lemma}\
\ref{composeisoms}\nno
&= \sum^{|\al|-2}[S_{\al_j}S_{\al_{j+1}}S_{\al_{j+2}}\cdots S_{\al_n}P_\mu S_{\al_n}^*
\cdots S_{\al_{j+2}}^*S_{\al_{j+1}}^*]\
+[S_{\al_{|\al|-1}}S_{\al_n}P_\mu S_{\al_n}^*] 
+ [S_{\al_n}P_\mu],\end{align*}
the last line following by induction. 
The application of Lemma \ref{composeisoms} requires 
$$ (S_{\ual} S_{\al_n}P_\mu S_{\al_n}^*)^*
(S_{\ual} S_{\al_n}P_\mu S_{\al_n}^*)
=S_{\al_n} P_\mu S_{\al_n}^*
=(S_{\al_n}P_\mu)(S_{\al_n}P_\mu)^*.$$
The second relation follows from Lemma \ref{composeisoms} 
also, since $S_\al^* S_\al=p_{r(\al)}=S_\beta^* S_\beta$.
\end{proof}

\begin{lemma} Let $A$ be the graph $C^*$-algebra of a locally finite
directed graph $E$ with no sources. Then for all edges $e\in E^1$, the
class $[S_e]\in K_0(M(F,A))$ is not zero. Similarly if $r(e)=s(\al)$
then $[S_eP_\al]\neq 0$.
\end{lemma}

\begin{proof} The assumptions on the graph ensure the existence of the
Kasparov module $(X,\D)$ constructed from the gauge action. The
pairing $\la[S_eP_\al],[(X,\D)]\ra$ is given by 
$[S_eP_\al S_e^*\Phi_0]=[S_eP_\al S_e^*]\in
K_0(F)$,
where $\Phi_0$ is the kernel projection of $\D$, whose range is the
trivial $F$-module $F$. This class is nonzero
since $F$ is an AF algebra, and so satisfies cancellation. 
\end{proof}

{\bf Remark}.
The hypothesis of `no sources' was introduced so that we could use the
nonzero index pairing to infer nonvanishing of the class
$[S_eP_\al]$. This restriction may be loosened provided we use
other ways of deducing the nonvanishing. For
instance, if the class 
$[P_\al]-[S_eP_\al S_e^*]=ev_*([S_eP_\al])\neq 0$ in
$K_0(F)$, then the class $[S_eP_\al]$ cannot be zero. On the other hand,
if $[P_\al]=[S_eP_\al S_e^*]$ in $K_0(F)$, then since $F$ is AF, there
exists a partial isometry $v\in F$ such that 
$S_eP_\al S_e^*=vv^*$ and
$P_\al=v^*v$. Then
$ u=1-P_\al+v^*S_eP_\al$
is a unitary, and so defines a class in $K_1(A)$. Since the map
$K_1(A)\to K_0(M(F,A))$ is an injection, and takes $[u]$ to $[S_eP_\al]$,
we would know that $[S_eP_\al]\neq 0$ if we knew that $[u]\neq 0$. 

\begin{cor}\label{rels} Let $A$ be the graph $C^*$-algebra of a
locally
finite connected
directed graph with no sources.
Two nonzero classes
$[S_e P_\al],\ [S_fP_\al]$, with $e,f$ edges in the graph 
and $\al$ an arbitrary path, are
equal  if and only if
$r(e)=r(f)$. Two nonzero 
classes $[S_e],\,[S_f]$, $r(e)$ a sink, are equal,
$[S_e]=[S_f]$, if and only if $r(e)=r(f)$.
\end{cor}

\begin{proof}
Suppose that $r(e)=r(f)$,
and that $[S_eP_\al]\neq 0$ (otherwise there is nothing to prove). Then as $S_e P_\al S_f^*\in F$ we have
$$
0=[S_eP_\al S_f^*]=[S_eP_\al]-[S_f P_\al],$$
by Lemma \ref{composeisoms}. 
 Conversely,
 if $r(e)\neq r(f)$ at least one of these classes is zero.

For the second statement we observe that if $r(e)=r(f)$ then 
$S_eS_f^*$ is nonzero, and then $[S_e]=[S_eS_f^*]+[S_f]=[S_f]$ by Lemma \ref{composeisoms}. If $r(e)\neq r(f)$, we suppose $[S_e]=[S_f]$, for a
contradiction, and compute the index pairing with the Kasparov module
$(X,\D)$  constructed from the gauge action.
The pairing is given by
$$\la [S_e],[(X,\D)]\ra=-[S_eS_e^*]=-[S_fS_f^*]=\la
[S_f],[(X,\D)]\ra.$$
Hence the class of $S_eS_e^*$ in $K_0(F)$ ($F$ is the fixed point
algebra) coincides with the class of $S_fS_f^*$. Since $F$ is an AF
algebra, there exists a partial isometry $v\in\mbox{span}\{S_\mu
  S_\nu^*:|\mu|=|\nu|\}$ such that $S_eS_e^*=vS_fS_f^*v^*$. Thus
$$ p_{r(e)}=S_e^*vS_f\,S_f^*v^*S_e=\sum_j
c_j\overline{c_k}S_e^*S_{\mu_j}S_{\nu_j}^*S_f\,
S_f^*S_{\nu_k}S_{\mu_k}^*S_e.$$
Here the paths $\mu_j$ start from $s(e)$ and end at some vertex $v_j$,
while the corresponding path $\nu_j$ starts from $s(f)$ and ends at
the same vertex $v_j$. Moreover there is at least one path $\mu_j$
with $S_e^*S_{\mu_j}\neq 0$ so $\mu_j=e\mu_{j_2}\cdots\mu_{j_{k}}$,
where  $|\mu_j|=k$.
However, $r(e)$ is a sink, so any such path is of the form
$\mu_j=e$. This forces the length of the corresponding $\nu_j$ to be
 $1$, and $\nu_j=f$. The only way the product
$S_{\mu_j}S_{\nu_j}^*=S_eS_f^*$ can now be non-zero is if $r(e)=r(f)$,
contradicting our assumption.
 \end{proof}

\begin{cor}\label{the-kicker}  Let $A$ be the graph $C^*$-algebra of a
locally finite connected
directed graph with no sources.
Then if two partial isometries of the form $[S_e],[S_f]$  satisfy
$[S_e]=[S_f]\in K_0(M(F,A))$ then there exists a partial isometry $\rho$
in $F$ such that $\rho S_e=S_f$ and $\rho^*\rho S_e=S_e=\rho^*S_f$.
\end{cor}

\begin{proof} The required partial
  isometry $\rho$ is
  $S_fS_e^*$. The remaining statements are immediate.
\end{proof}

\begin{lemma}\label{edges-gen}
Let $E$ be a row-finite directed graph. Then the group
$K_0(M(C^*(E)^\gamma,C^*(E)))$ is generated by the classes $[S_e P_\al]$, where
$e$ is an edge  and  $\al$ is a path.
\end{lemma}

\begin{proof} Let $[v]\in K_0(M(C^*(E)^\gamma,C^*(E)))$ and consider
$$ ev_*[v]=[v^*v]-[vv^*]\in K_0(C^*(E)^\gamma).$$
Now $K_0(C^*(E)^\gamma)$ is generated by the classes $[p_\mu]$, $p_\mu=S_\mu
S_\mu^*$, where $\mu\in E^*$ is a path, \cite{PR}. As $C^*(E)^\gamma$
is an AF algebra, there are partial isometries $W, Z$ over
$C^*(E)^\gamma$ such that
\begin{equation} W^*W=v^*v,\ \ WW^*=\sum_jp_{\mu_j},\ \ \ \ ZZ^*=vv^*,\ \
Z^*Z=\sum_kp_{\nu_k},\label{proj-sums}\end{equation}
and $[v]=[Z^*vW^*]$. The latter follows because $Z,\,W$ are partial isometries
over $F$ and so represent zero, while $[Z^*vW^*]=[Z^*]+[v]+[W^*]$.
In Equation \eqref{proj-sums} the sums are necessarily orthogonal, and may
be in a matrix algebra over $C^*(E)^\gamma$, and some zeroes (place-holders to
make the matrix dimensions equal) may have
been omitted from the sums. Observe that
$ev_*[Z^*vW^*]=\sum_k[p_{\nu_k}]-\sum_j[p_{\mu_j}]$. By considering
$p_{\nu_k}Z^*vW^*p_{\mu_j}$ we may suppose without loss of generality that we
have only one summand so that $WW^*=p_\mu$ and
$Z^*Z=p_\nu$. Then
$$ ev_*[Z^*vW^*S_\mu S_\nu^*]=[p_\nu]-[p_\nu]=0.$$
Hence $[v]=[Z^*vW^*]=[S_\nu S_\mu^*]$ modulo the image of $i_*$, and Lemma
\ref{gens} completes
the proof for $[v]\not\in\mbox{Image}(i_*)$. Observe that $S_\nu S_\mu^*\neq 0$
(and so $r(\mu)=r(\nu)$)
is a consequence.

In the case $ev_*[v]=0$, so that $[v]\in \mbox{Image}(i_*)$ we observe that
there is a partial isometry $X$ over $C^*(E)^\gamma$ such that
$X^*X=v^*v$ and $XX^*=vv^*$ so that
$ 1-v^*v+X^*v$
is unitary. Then, again since all partial isometries are over $F$,
$$[v]=[WX^*vW^*]=[WX^*ZZ^*vW^*]=[WX^*ZS_\nu S_\mu^*]=i_*[1-p_\mu +WX^*Z
S_\nu S_\mu^*]$$
gives a unitary representative of $v$. Since $i_*[1-p_\mu+WX^*Z S_\nu
S_\mu^*] =[S_\nu S_\mu^*]$, Lemma \ref{gens} completes the proof.
\end{proof}

The structure of $K_1(M(F,A))$ is even simpler.
\begin{lemma} If $E$ is a row-finite directed graph, $A=C^*(E)$ and
$F=C^*(E)^\gamma$, then $K_1(M(F,A))=0$.
\end{lemma}

\begin{proof} The exact sequence
$0\to A\otimes C_0(0,1)\to M(F,A)\to F\to 0$
and $K_1(F)=0$ yields
\begin{equation}
0\to K_1(A)\to K_0(M(F,A))\stackrel{ev_*}{\to}
K_0(F)\to K_0(A)\to K_1(M(F,A))\to0.
\label{eq:ker-coker}
\end{equation}
By Lemma \ref{lm:j}, the map $K_0(F)\to K_0(A)$ is induced 
(up to sign and
Bott periodicity)  by inclusion $j:F\to A$. This
map is surjective on $K_0$ by \cite{PR}[Lemma 4.2.2], and so $K_1(M(F,A))=0$.
\end{proof}

In \cite{PR}, the $K$-theory of a graph algebra $C^*(E)$, 
where $E$ has no sources or sinks, was computed as the
kernel ($K_1$) and cokernel ($K_0$) of the map given by the vertex
matrix on $\Z^{E^0}$ (there are subtleties when sinks are
involved). The proof of this result involves the dual of the gauge
action and the Pimsner-Voiculescu exact sequence for crossed products.
In Equation \eqref{eq:ker-coker}, we see the $K$-theory
again expressed as the kernel and cokernel of a map, but this time it
arises with no serious effort. The difference of course is that the
groups $K_0(M(F,A))$ and $K_0(F)$ are in general harder to compute.
 
While the map $ev_*:K_0(M(F,A))\to K_0(F)$ is neither one-to-one nor
onto in general, we can deduce that the two groups $K_0(M(F,A))$ and
$K_0(F)$ are in fact isomorphic in a wide range of examples.
We let $(\hat{X},\hat\D)$ be the APS Kasparov module arising from 
the Kasparov module $(X,\D)$. 

\begin{prop}
\label{pr:wow}
 Let $A$ be the graph $C^*$-algebra of a
 locally finite connected
directed graph with no sources and no sinks. Then the map 
$ {\rm Index}_{\hat\D}:K_0(M(F,A))\to K_0(F)$
given by the Kasparov product with the Kasparov module of the gauge
action is an isomorphism.
\end{prop}
\begin{proof}
First the index map is a well-defined homomorphism, \cite{K}. We begin
by showing that the index map is one-to-one. So suppose that we have
edges $e,\ g$ and paths $\al,\ \beta$ in our graph (with no range a sink), and suppose
that 
$\mbox{Index}_{\hat\D}([S_eP_\al])
=\mbox{Index}_{\hat\D}([S_gP_\beta])$. A simple computation 
using Theorem 5.1 yields
$$ \mbox{Index}_{\hat\D}([S_eP_\al])=[S_eP_\al S_e^*]=
[S_gP_\beta S_g^*]=\mbox{Index}_{\hat\D}([S_gP_\beta]).$$
As $F$ is an AF algebra, we can find a partial isometry $v$ in $F$
such that 
$$S_eP_\al S_e^*=vS_gP_\beta S_g^*v^*.$$
Then setting $w=P_\al S_e^*vS_g P_\beta\neq 0$ we have 
$$P_\al=ww^*=wP_\beta w^*\ \ \mbox{and}\ \ 
P_\beta=w^*w=w^*P_\al w.$$ 
We will use 
Lemma \ref{composeisoms} below and need to check that some partial isometries have the same source projections. First
observe that 
$(S_eP_\al wP_\beta)^*(S_eP_\al wP_\beta)=P_\beta=w^*w$, so
$$[S_eP_\al]=[S_eP_\al wP_\beta w^*]=[S_eP_\al wP_\beta]+[w^*]=[S_eP_\al wP_\beta]=[S_eP_\al S_e^*vS_gP_\beta],$$
the second last last equality following since $w$ is a partial isometry in $F$. Now since 
$(S_gP_\beta)(S_gP_\beta)^*=S_gP_\beta S_g^*$ and 
$(S_eP_\al S_e^*v)^*(S_eP_\al S_e^*v)=S_gP_\beta S_g^*$, 
we can apply Lemma \ref{composeisoms} again to find
$$[S_eP_\al]=[S_eP_\al S_e^*vS_gP_\beta]=[S_eP_\al S_e^*v]
+[S_gP_\beta]=[S_gP_\beta].$$
Thus $\mbox{Index}_{\hat\D}$ is one-to-one.
Now supposing that our graph has no sinks, every class in $K_0(F)$ is
a sum of classes $[p_\mu]=[S_\mu S_\mu^*]$, where $\mu$ is a path in
the graph of length at least one.
For a given $\mu=\mu_1\cdots\mu_{|\mu|}$, define
$\overline{\mu}=\mu_2\cdots\mu_{|\mu|}$. Then it is straightforward to
check that
$$\mbox{Index}_{\hat\D}([S_\mu S_{\overline{\mu}}^*])=[p_\mu].$$
Hence the index map is onto and we are done.
\end{proof}

Observe that this does not mean that the $K$-theory of the graph
algebra is zero! The evaluation map and the index map are very
different. For the Cuntz algebra $O_n$, $n\geq 2$, 
for example, the fixed point
algebra has $K$-theory $K_0(F)\cong\Z[1/n]$ and so we have
$$ev_*([S_\mu])=[1]-[S_\mu S_\mu^*]\sim
1-\frac{1}{n^{|\mu|}}=(n^{|\mu|}-1)\frac{1}{n^{|\mu|}},$$
with $\ker(ev_*)\cong K_1(O_n)=0$ and 
$\mbox{coker}(ev_*)\cong K_0(O_n)=\Z_{n-1}$.
The index map gives us
$$\mbox{Index}_{\hat\D}([S_\mu])=
\sum_{j=0}^{|\mu|-1}[S_\mu S_\mu^*\Phi_j].$$
This equality follows from Theorem 5.1, and to determine the right hand side more explicitly, set $\overline{\mu}=\mu_{j+1}\cdots\mu_{|\mu|}$
and define the partial isometry 
$W=S_\mu S_{\overline{\mu}}^*\Phi_0$. Then 
$WW^*=S_\mu S_\mu^*\Phi_j$ and 
$W^*W=S_{\overline{\mu}}S_{\overline{\mu}}^*\Phi_0$. Thus in 
$K_0(F)$ we have 
$$\mbox{Index}_{\hat\D}([S_\mu])=
\sum_{j=0}^{|\mu|-1}[S_\mu S_\mu^*\Phi_j]=
\sum_{j=0}^{|\mu|-1}[S_{\overline{\mu}}S_{\overline{\mu}}^*\Phi_0]
=\sum_{j=0}^{|\mu|-1}[S_{\overline{\mu}}S_{\overline{\mu}}^*]
\sim\sum_{j=0}^{|\mu|-1}n^{-(|\mu|-j)}
=
\left(\frac{n^{|\mu|}-1}{n-1}\right)\frac{1}{n^{|\mu|}}.$$

The evaluation map and the mapping cone exact sequence gives us
$K_0(M(O_n^\gamma,O_n))\cong (n-1)\Z[1/n]$ (those polynomials all of
whose coefficients have a factor of $n-1$) which is of course
isomorphic to $\Z[1/n]\cong K_0(F)$ as an additive group.

\end{document}